\documentclass{article}
\usepackage{amsmath}
\usepackage{amssymb}
\usepackage{mathabx}
\usepackage{amsthm}
\usepackage{mathrsfs}
\usepackage[all]{xy}
\usepackage{cite}
\usepackage{graphicx}
\usepackage{subfig}
\usepackage{float}
\usepackage{enumitem}
\usepackage{pifont}
\usepackage{xfrac}
\usepackage{appendix}
\setitemize[1]{itemsep=0pt,parsep=\parskip,topsep=0pt}
\begin{document}
\theoremstyle{definition}
\newtheorem*{theorem}{Theorem}
\newtheorem{lemma}{Lemma}[section]
\newtheorem{proposition}[lemma]{Proposition}
\newtheorem{corollary}[lemma]{Corollary}
\newtheorem{propdefn}[lemma]{Proposition and Definition}
\newtheorem*{definition}{Definition}
\newtheorem{convention}{Convention}
\newtheorem*{fact}{Fact}
\newtheorem{example}{Example}
\theoremstyle{remark}
\newtheorem*{remark}{Remark}
\title{\textbf{Isoparametric foliations of $S^n\times S^n$ of codimension $1$}}
\author{Teng Wang}
\date{}
\maketitle

\abstract{In this article, we give several families of new examples of isoparametric foliations of $S^n\times S^n$;
establish some basic properties of isoparametric foliations of $S^n\times S^n$ of codimension $1$,
especially any of them can be realized as level sets of the restriction of
a bi-homogeneous polynomial of bi-degree $(1,1),(2,2),(3,3),(4,4)$ or $(6,6)$ to $S^n\times S^n$;
and classify bi-homogeneous isoparametric polynomials of bi-degree $(1,1)$ and $(2,2)$
and isoparametric foliations on $S^n\times S^n$ with $n$ even and not equal to $4$. }

		
		\bigskip

\section{Introduction}

Let $M$ be a smooth manifold and $\mathcal{F}$ a disjoint union of connected immersed submanifolds of $M$
which are called leaves. We say $\mathcal{F}$ is a foliation of $M$ if for any $L\in\mathscr{F}$ and $p\in L$,
$$T_pL=\{X_p|X\text{ is a smooth vector field on $N$ tangent to each leaf of }\mathscr{F}\}.$$
A foliation is regular if all its leaves are of the same dimension, singular otherwise.
A codimension-$1$ foliation on a Riemannian manifold is isoparametric if its leaves
are equidistant and have constant mean curvature.
The most classical example of isoparametric foliations in the physical world is
the collection of the wavefronts of a propagating wave with a uniform source density.

A function $F$ on a Riemannian manifold with $|\mathrm{grad}\,F|^2$ and $\Delta\, F$ both functions of $F$
is called an isoparametric function. A hypersurface $M$ on a Riemannian manifold is isoparametric if
it is a level set of an isoparametric function defined on a neighbourhood of $M$,
or equivalently, a leaf of an isoparametric foliation defined on a neighbourhood of $M$.

The three concepts isoparametric foliation of codimension $1$, connected closed isoparametric hypersurface, isoparametric function
are essentially equivalent when the ambient space is a space form,
for they can be constructed from each another in the following ways:
\begin{itemize}
\item[1.] a regular level set of an isoparametric function is a closed isoparametric hypersurface;
\item[2.] connected components of level sets of an isoparametric function make up an isoparametric foliation;
\item[3.] fix the leaf $L$ of an isoparametric foliation, the distance function $\mathrm{dist}(L,\cdot)$
on the ambient space is an isoparametric function;
\item[4.] the family of submanifolds obtained by parallel translation of a connected closed isoparametric hypersurface
forms an isoparametric foliation.
\end{itemize}

When the ambient space is a general Riemannian manifold, 4.) may not hold since
the isoparametric property of a hypersurface may not be preserved by parallel translation.
However, by Lemma \ref{isoparabebihomo}, when the ambient space is $S^n\times S^n$,
which we mainly talk about in this article, 4.) holds, and the three concepts are essentially equivalent.

Isoparametric hypersurfaces has been studied as a mathematical object since 1919 \cite{Somigliana1918Huygens}.
Initially, mathematicians are only interested in isoparametric hypersurfaces
in space forms, which originate from physics naturally.
The classification of isoparametric hypersurfaces in Euclidean and real hyperbolic spaces
has not taken too much effort, and is completed by
Segre \cite{segre1938famiglie} and Cartan \cite{cartan1938hyperbolic}, respectively.
However, isoparametric hypersurfaces in spheres are far more complicated,
and its study has engaged mathematicians for generations.
Cartan has studied isoparametric hypersurfaces in space forms and proven many basic properties of them
in \cite{cartan1938hyperbolic,cartan1939sphere31,cartan1939sphere32,cartan1940sphere4},
especially a hypersurface in a space form is isoparametric if and only if its principal curvatures are constant.

In \cite{cartan1939sphere31}, Cartan has asked three questions for isoparametric hypersurfaces in spheres
which is instructive to the subsequent research:
\begin{itemize}
\item[1.] For any positive integer $g$, dose there exist an isoparametric hypersurface with $g$ distinct principal curvatures?
\item[2.] Does there exist an isoparametric hypersurface with principal curvatures of different multiplicities.
\item[3.] Is every isoparametric hypersurface homogeneous?
\end{itemize}
Nomizu \cite{nomizu1973some} has given out examples with four distinct principal curvatures of multiplicities
$m_1=m_3=m$, $m_2=m_4=1$. That gives an affirmative answer to Cartan's second question.
Hsiang and Lawson \cite{hsiang1971minimal}, and Takagi \cite{takagi1972principal} have classified
homogeneous isoparametric hypersurfaces in spheres.
Ozeki and Takeuchi \cite{OT1975eg1,OT1976eg2} have constructed
examples with four distinct pricipal curvatures using representations of Clifford algebras.
Some of their examples are not included in the list of homogeneous ones.
That gives a negative answer to Cartan's third question.
Ferus, Karcher and M{\"u}nzner \cite{FKM1981eg} have constructed infinitely many families of examples
with $g=4$ using representations of Clifford algebras,
$$F(x)=\sum_{i=1}^m\langle P_ix,x\rangle^2,$$
where $e_i\mapsto P_i$ ($i=1,\dots,m$) stands for a representation of the Clifford algebra $Cl_m^*$,
which is a generalization of Ozeki and Takeuchi's construction, and called FKM-type.
M{\"u}nzner \cite{munzner1980isoparametrische},\cite{munzner1981isoparametrische} have established many important properties of isoparametric hypersurfaces in spheres, representatively,
\begin{itemize}
\item[1.] the pricipal curvatures are $\lambda_\alpha=\cot\theta_\alpha$ for $\alpha=1,\dots,g-1$ with
$0\le\theta_1\le\cdots\le\theta_{\alpha-1}<\pi$ and $\theta_\alpha=\theta_0+\frac{\alpha\pi}{g}$,
its multiplicities satisfies $m_\alpha=m_{\alpha+2}$ (indices modulo $g$);
\item[2.] any isoparametric hypersurface is a level set of the restriction of a Cartan-M{\"u}nzner polynomial to $S^n$;
\item[3.] the number of distinct principal curvatures can only be 1,2,3,4 or 6.
\end{itemize}
The item $3.)$ above gives a negative answer to Cartan's first question.

Besed on the above results, it is reasonable to conjecture that
an isoparametric hypersurface in a sphere is either homogeneous or of FKM-type.
In fact, that is
proven to be true after decades of effort by numerous mathematicians.
The case $g=1,2,3$ is proven by Cartan \cite{cartan1939sphere31,cartan1939sphere32}.
For $g=4$, Takagi \cite{takagi1976class}; Ozeki and Takeuchi \cite{OT1975eg1,OT1976eg2};
Cecil, Chi and Jensen \cite{cecil2007classification4};
Chi \cite{chi2011classification4,chi2013classification4,chi2020classification4}
have proven the cases $m_1=1$; $m_1=2$; $m_2\ge2m_1-1$; and $(m_1,m_2)=(3,4)(4,5),(6,9),(7,8)$, respectively.
For $g=6$, Abresch \cite{abresch1983isoparametric} has showed that $m_1=m_2=1$ or $m_1=m_2=2$,
the former case is proven by Dorfmeister and Neher \cite{dorfmeister1985classification6},
and the latter case by Miyaoka \cite{miyaoka2013classification6,miyaoka2016errata}.

Besides the case of space forms, isoparametric hypersurfaces in symmetric spaces
is also of great interest to mathematicians.
But the classification of isoparametric hypersurfaces in symmetric spaces, even rank-$1$ symmetric spaces
 is not completed yet.
A great progress is made by Dom{\'\i}nguez-V{\'a}zquez and his cooperators.
They classified isoparametric hypersurfaces in $\mathbb{C}P^n$, $\mathbb{H}P^n$
and $\mathbb{C}H^n$ in a series of articles
\cite{dominguez2016isoparametric,dominguez2018polar,dominguez2025isoparametric,diaz2017isoparametric}.

Isoparametric theory are also extended to higher codimensions and more general ambient spaces:
Wang has studied isoparametric functions on general Riemannian manifolds in \cite{wang1987isoparametric}.
Harle \cite{harle1982isoparametric}, Carter and West \cite{carter1985isoparametric},
Terng \cite{terng1985isoparametric}, and etc. have respectively given the definition of
isoparametric submanifold and isoparametric map in space forms,
which are generalizations of isoparametric hypersurface and isoparametric function, respectively.
Terng and Thorbergsson \cite{terng1995submanifold} also generalized the concept isoparametric submanifold
to equifocal submanifold which is defined in symmetric spaces.
For more study on isoparametric theory on general Riemannian manifolds,
see \cite{ge2010isoparametric,ge2014geometry,miyaoka2013transnormal}.

An isoparametric submanifold of $\mathbb{R}^N$ or $\mathbb{H}^N$ of codimension $r$ is
essentially an isoparametric submanifold in $S^{N-1}$ of codimension $(r-1)$
(see \cite{terng1985isoparametric,wu1992isoparametric} for a proof).
Thorbergsson \cite{thorbergsson1991isoparametric} and Olmos \cite{olmos1993isoparametric} have proven that
an isoparametric foliation of a sphere of codimension $\ge2$ with compact full irreducible regular leaves is homogeneous.
Therefore, the classification of isoparametric submanifolds of space forms of any codimension is reduced to
the classification of isoparametric hypersurfaces in spheres and isometry group actions on spheres.
Christ \cite{christ2002homogeneity} has proven that a complete connected irreducible equifocal submanifold
of codimension $\ge2$ of a simply connected compact symmetric spaces is homogeneous,
which reduces the classification of equifocal submanifolds of compact symmetric spaces to
the codimension-$1$ case and the homogeneous case.
Therefore, it is reasonable to believe that the codimension-$1$ case is a crucial part of
the classification project in isoparametric theory.

Urbano \cite{urbano2016S2S2} is the first to study isoparametric hypersurfaces in $S^2\times S^2$,
and classified them. In his article, Urbano defined the product structure $P(X,Y)=(X,-Y)$,
the angle function $C=\langle PN,N\rangle$ and a vector field $V=PN-CN$.
Many subsequent articles is inspired by those concepts.

Gao, Ma and Yao have classified isoparametric hypersurfaces in $\mathbb{H}^2\times\mathbb{H}^2$ in \cite{gao2024hypersurfaces},
and then in \cite{gao2023productspace} isoparametric hypersurfaces in a product space
$M_{\kappa_1}\times M_{\kappa_2}$,
where $M_\kappa$ is the $2$-dimensional space form of constant curvature $\kappa\in\{-1,0,1\}$.
De Lima and Pipoli \cite{de2024isoparametric} have also classified isoparametric hypersurfaces
in $S^n\times\mathbb{R}$ and $\mathbb{H}^n\times\mathbb{R}$.
Tan, Xie and Yan \cite{tan2025isoparametric} have enhanced De Lima and Pipoli's result,
classified isoparametric hypersurfaces in $S^n\times\mathbb{R}^m$ and $\mathbb{H}^n\times\mathbb{R}^m$
with constant angle function.
De Lima and Pipoli \cite{de2025isoparametric} have proven any connected isoparametric hypersurface of
$\mathbb{Q}^{n_1}_{\epsilon_1}\times\mathbb{Q}^{n_2}_{\epsilon_2}$ has constant angle function,
and classified homogeneous isoparametric hypersurfaces in $\mathbb{Q}^{n_1}_{\epsilon_1}\times\mathbb{Q}^{n_2}_{\epsilon_2}$
satisfying a one-point condition, where $\mathbb{Q}^n_\epsilon$ is the $n$-dimensional space form of constant curvature $\epsilon\in\{-1,0,1\}$.
Mafio, JBM. dos Santos, JP. dos Santos and Van der Veken \cite{manfio2025hypersurfaces} have classified hypersurfaces in
$S^3\times\mathbb{R}$ and $\mathbb{H}^3\times\mathbb{R}$ with constant principal curvatures,
and proven they are all isoparametric.
For more study on submanifold geometry on product spaces, see \cite{li2024hypersurfaces,gao2022real,chen2021minimal}.

For detailed surveys on isoparametric theory, see \cite{chi2020isoparametric,ge2025overview}.

In this article, we extend the inquiry of Urbano and study isoparametric foliations of $S^n\times S^n$ for arbitrary $n$,
and take a step forward on their classification.

Besides constant functions, the most straightforward example of isoparametric function on $S^n\times S^n$ is $F(x,y)=f(x)$ or $F(x,y)=f(y)$
with $f$ an isoparametric function on $S^n$.
These examples are essentially  isoparametric functions on $S^n$, and we call them trivial.
More precisely, for an isoparametric function $F$ on $S^n\times S^n$, if $F(x,y)=f(x),f(y)$ or $c$
for some function $f$ on $S^n$ or $c\in\mathbb{R}$, we call $F$ and the isoparametric foliation of its level sets trivial.
In the following sections, isoparametric functions and isoparametric foliations on $S^n\times S^n$
are assumed to be nontrivial, unless otherwise specified.

The main result of this article is as follows.
\begin{theorem}
\begin{itemize}
\item[(A)]A nontrivial isoparametric foliation for $S^n\times S^n$ is determined by
the restriction of a bi-homogeneous polynomial of bi-degree $(g,g)$, $g=1,2,3,4$ or $6$, to $S^n\times S^n$,
whose restirction to any section $S^n\times\{y\}$ or $\{x\}\times S^n$ is an isoparametric polynomial for $S^n$ of degree $g$;
\item[(B)]a nontrivial isoparametric foliation of $S^n\times S^n$ of bi-degree $(1,1)$ is the collection of level sets of
$$F(x,y)=\langle x,y\rangle,$$
up to isometric transformations of $S^n\times S^n$;
\item[(C)]a nontrivial isoparametric foliation of $S^n\times S^n$ of bi-degree $(2,2)$ is the collection of level sets of
\begin{equation}\label{cuieg}
F(x,y)=\langle x,y\rangle^2+\sum_{i=1}^{m-1}\langle x,E_iy\rangle^2,
\end{equation}
up to isometric transformations of $S^n\times S^n$,
where $E_i=\phi(e_i)$ for $i=1,\dots,m-1$, $\phi$ is a representation of the Clifford algebra $Cl_{m-1}$.
\end{itemize}
\end{theorem}

As a corollary of the Theorem, when $n$ is an even numbers $\ne4$ or $n=3$, a nontrivial isoparametric foliation of $S^n\times S^n$ is
one of those ones given in (B) or (C) of the theorem.

It is worth noting that the Clifford examples are initially given in the form of
(\ref{fulleg}) and (\ref{notfulleg}) in Subsection \ref{exampleclifford}.
Cui is the first to think of Clifford examples using representations of $Cl_m$,
and write down the formula (\ref{cuieg}) without knowing (\ref{fulleg}) and (\ref{notfulleg}). 
Cui's formula turns out to be the same as
(\ref{fulleg}) and (\ref{notfulleg}) up to congruence with respect to $S^n\times S^n$.
See Lemma \ref{cliffordegeq} for a proof of that.

Cui \cite{cui2025cone} has used minimal level sets of (\ref{cuieg})
to construct infinity many families of new area-minimizing cones of codimension $2$.

This article is organized as follows:
in Section \ref{prelimilaries}, we introduce some basic definitions, conventions and results used in subsequent sections;
in Section \ref{eg}, we give some examples of isoparametric functions on $S^n\times S^n$;
in Section \ref{property}, we prove (A) of the Theorem in steps,
establish ways to compute the second fundamental forms and the principal curvatures
of isoparametric hypersurfaces in $S^n\times S^n$,
and inquire homogeneity and integral submanifolds of principal distributions;
in Section \ref{classification}, we classify nontrivial isoparametric foliations of $S^n\times S^n$ of bi-degree $(1,1)$ and $(2,2)$.

The author here acknowledge Professor Xiaowei Xu for providing the problem of classifying
isoparametric foliations of $S^n\times S^n$ and his visionary guidance,
and Hongbin Cui for his instructive discussions and invaluable advices,
especially proposing the unified formula (\ref{cuieg}) of all Clifford examples.
This work is supported by the NSFC (No.12571057),
the project of Stable Support for Youth Team in Basic Research Field, CAS (YSBR-001).

This article was submitted to the Journal of Geometric Analysis on 14 January 2026 with submission ID JGEA-D-26-00069. 
Compared with that version, in this version of the article, we corrected some typographical errors, restated some sentences for clarity, and replaced the latter part of the proof to Lemma \ref{prinfoli} with a new proof. 

\section{Preliminaries}\label{prelimilaries}

This section is divided into four parts, basic properties of bi-homogeneous polynomial,
evolution equations of hypersurfaces moving in parallel,
the Cartan-M\"{u}nzner equations,
and some basic properties of representations of Clifford algebras.
Although most the results in this section are not new,
we present proofs or computations of them for completeness.

Isoparametric functions on $S^n$ and $S^n\times S^n$ are both included in this article.
To distinguish them, we adopt the following conventions.

For a function $F$ defined on $\mathbb{R}^{n+1}$ (resp. $\mathbb{R}^{n+1}\times\mathbb{R}^{n+1}$),
we call it isoparametric for $S^n$ (resp. $S^n\times S^n$)
if $F|_{S^n}$ (resp. $F|_{S^n\times S^n}$) is isoparametric.

We hat operators on $\mathbb{R}^{n+1}$, $S^n$ (or $S^n\times S^n$) and a hypersurface $M$ in $S^n$ (or $S^n\times S^n$)
with nothing, ``$\ \widetilde{\,}\ $'' and ``$\ \widehat{\,}\ $'' respectively to distinguish them,
except in Subsection \ref{ODE} and Subsection \ref{principalprop}.

Let $F$ and $F'$ be functions defined on $\mathbb{R}^{n+1}$ (resp. $\mathbb{R}^{n+1}\times\mathbb{R}^{n+1}$),
we say $F$ and $F'$ are congruent to each other with respect to $S^n$ (resp. $S^n\times S^n$)
if there exists $A\in O(n+1)$ (resp. $\begin{bmatrix}O(n+1)&\\&O(n+1)\end{bmatrix}\rtimes
\{\begin{bmatrix}I&\\&I\end{bmatrix},\begin{bmatrix}&I\\I&\end{bmatrix}\}$)
such that $F(X)=F'(AX)$ for any $X\in\mathbb{R}^{n+1}$ (resp. $\mathbb{R}^{2n+2}$).

\subsection{Bi-homogeneous polynomials}

We say a function $f:\mathbb{R}^{n+1}\times\mathbb{R}^{n+1}$ is bi-homogeneous of bi-degree $(r,s)$
if $f(cx,y)=c^rf(x,y)$ and $f(x,cy)=c^sf(x,y)$ for any $(x,y)\in \mathbb{R}^{n+1}\times\mathbb{R}^{n+1}$ and
$c\in\mathbb{R}^+$.

\begin{lemma}\label{nablabi}
Let $F:\mathbb{R}^{n+1}\times\mathbb{R}^{n+1}$ be a bi-homogeneous function of bi-degree $(r,s)$, then
on $S^n\times S^n$ we have
\begin{itemize}
\item[1.]$|\widetilde{\mathrm{grad}}\,F|^2=|\mathrm{grad}\,F|^2-(r^2+s^2)F^2$,
\item[2.]$\widetilde{\Delta}\,F=\Delta\,F-((r+n-1)r+(s+n-1)s)F$.
\end{itemize}
\end{lemma}

\begin{proof}
Let $\mathrm{grad}^xF,\mathrm{grad}^yF,\widetilde{\mathrm{grad}}^xF$ and $\widetilde{\mathrm{grad}}^yF$
be the gradient operators on $\mathbb{R}^{n+1}\times\{y\},\{x\}\times\mathbb{R}^{n+1},S^n\times\{y\}$ and $\{x\}\times S^n$, respectively.
Let $\Delta^xF,\Delta^yF,\widetilde{\Delta}^xF$ and $\widetilde{\Delta}^yF$
be the Laplacian operations on $\mathbb{R}^{n+1}\times\{y\},\{x\}\times\mathbb{R}^{n+1},S^n\times\{y\}$ and $\{x\}\times S^n$, respectively.

Since the restriction of $F$ to each section $S^n\times\{y\}$ (resp. $\{x\}\times S^n$) is of degree $r$ (resp.$s$),
by Theorem 3.30 in \cite{thebook}, we have
\begin{gather*}
|\widetilde{\mathrm{grad}}^xF|^2=|\mathrm{grad}^xF|^2-r^2F^2,\quad
|\widetilde{\mathrm{grad}}^yF|^2=|\mathrm{grad}^yF|^2-s^2F^2,\\
\widetilde{\Delta}^xF=\Delta^xF-(r-n+1)rF,\quad
\widetilde{\Delta}^yF=\Delta^yF-(s-n+1)sF,
\end{gather*}
and the lemma follows.
\end{proof}

By Lemma \ref{nablabi}, the restriction of a bi-homogeneous function $F$ on $\mathbb{R}^{n+1}\times\mathbb{R}^{n+1}$
to $S^n\times S^n$ is isoparametric if and only if $|\mathrm{grad}\,F|^2$ and $\Delta\,F$ both take
constant values on each level set of $F|_{S^n\times S^n}$. We also call such  bi-homogeneous functions isoparametric.

\subsection{Ricatti equation and curvature-adaptedness equation}\label{ODE}

In this subsection, we consider the movement of a hypersurface in a Riemannian manifold
along its unit normal vector field, and compute its variation formulae.
Let M be a hypersurface in a Riemannian manifold $N$, $F:M\times[0,\varepsilon)\to N$ such that
$F(x,0)=x$ and $F_*(\frac{\partial}{\partial t})|_{t=\tau}$ is the unit normal vector field to $M_\tau:=F(M\times\{\tau\})$.
By the theory of first-order PDE,there exists a unique such $F$.
Let $(x^1,\dots,x^n)$ be a local coordinate of $M$,
then there always holds $F_*(\frac{\partial}{\partial t})\perp F_*(\frac{\partial}{\partial x^i})$.
We abbreviate $F_*(\frac{\partial}{\partial t})$ and $F_*(\frac{\partial}{\partial x^i})$ to
$\frac{\partial}{\partial t}$ and $\frac{\partial}{\partial x^i}$, respectively.
Denote the Levi-Civita connection of $N$ by $\nabla$,
the Riemannian curvature tensor of $N$ by $R$,
and the second fundamental form of $M_t$ by $h_{ij}dx^i\otimes dx^j$, where
$h_{ij}=\langle\nabla_{\frac{\partial}{\partial x^i}}\frac{\partial}{\partial x^j},\frac{\partial}{\partial t}\rangle$.

From
\begin{gather*}
\langle\nabla_\frac{\partial}{\partial t}\frac{\partial}{\partial t},\frac{\partial}{\partial x^i}\rangle
=-\langle\frac{\partial}{\partial t},\nabla_\frac{\partial}{\partial t}\frac{\partial}{\partial x^i}\rangle
=-\langle\frac{\partial}{\partial t},\nabla_\frac{\partial}{\partial x^i}\frac{\partial}{\partial t}\rangle
=-\frac{1}{2}\frac{\partial}{\partial x^i}|\frac{\partial}{\partial t}|^2=0\\
\text{and }\langle\nabla_\frac{\partial}{\partial t}\frac{\partial}{\partial t},\frac{\partial}{\partial t}\rangle
=\frac{1}{2}\frac{\partial}{\partial t}|\frac{\partial}{\partial t}|^2=0,
\end{gather*}
we know that $\nabla_\frac{\partial}{\partial t}\frac{\partial}{\partial t}=0$,
which means that the normal distribution of a transnormal foliation is geodesic.
From
$\langle\frac{\partial}{\partial t},\nabla_\frac{\partial}{\partial x^i}\frac{\partial}{\partial t}\rangle=0$
and $\\ \langle\frac{\partial}{\partial x^j},\nabla_\frac{\partial}{\partial x^i}\frac{\partial}{\partial t}\rangle=-h_{ij}$
we know that $\nabla_\frac{\partial}{\partial t}\frac{\partial}{\partial x^i}=-h_{ik}g^{kl}\frac{\partial}{\partial x^l}$.
We compute the variation formulae
\begin{gather*}
\begin{aligned}
\frac{\partial g_{ij}}{\partial t}=&\frac{\partial}{\partial t}\langle\frac{\partial}{\partial x^i},\frac{\partial}{\partial x^j}\rangle\\
=&\langle\nabla_\frac{\partial}{\partial t}\frac{\partial}{\partial x^i},\frac{\partial}{\partial x^j}\rangle
+\langle\frac{\partial}{\partial x^i},\nabla_\frac{\partial}{\partial t}\frac{\partial}{\partial x^j}\rangle\\
=&\langle\nabla_\frac{\partial}{\partial x^i}\frac{\partial}{\partial t},\frac{\partial}{\partial x^j}\rangle
+\langle\frac{\partial}{\partial x^i},\nabla_\frac{\partial}{\partial x^j}\frac{\partial}{\partial t}\rangle\\
=&-\langle\frac{\partial}{\partial t},\nabla_\frac{\partial}{\partial x^i}\frac{\partial}{\partial x^j}\rangle
-\langle\nabla_\frac{\partial}{\partial x^j}\frac{\partial}{\partial x^i},\frac{\partial}{\partial t}\rangle\\
=&-2h_{ij},
\end{aligned}\\
0=\frac{\partial g^{ij}g_{jk}}{\partial t}=-2g^{ij}h_{jk}+\frac{\partial g^{ij}}{\partial t}g_{jk}
\Longrightarrow\frac{\partial g^{il}}{\partial t}=2g^{ij}h_{jk}g^{kl},\\
\begin{aligned}
\frac{\partial h_{ij}}{\partial t}=&\frac{\partial}{\partial t}
\langle\nabla_\frac{\partial}{\partial x^i}\frac{\partial}{\partial x^j},\frac{\partial}{\partial t}\rangle\\
=&\langle\nabla_\frac{\partial}{\partial t}
\nabla_\frac{\partial}{\partial x^i}\frac{\partial}{\partial x^j},\frac{\partial}{\partial t}\rangle\\
=&R_{0i0j}+\langle\nabla_\frac{\partial}{\partial x^i}
(\nabla_\frac{\partial}{\partial t}\frac{\partial}{\partial x^j}),\frac{\partial}{\partial t}\rangle\\
=&R_{0i0j}-\langle\nabla_\frac{\partial}{\partial t}\frac{\partial}{\partial x^i}
,\nabla_\frac{\partial}{\partial t}\frac{\partial}{\partial x^j}\rangle\\
=&R_{0i0j}-\langle-h_{jk}g^{kl}\frac{\partial}{\partial x^l}
,-h_{ip}g^{pq}\frac{\partial}{\partial x^q}\rangle\\
=&R_{0i0j}-h_{ik}g^{kl}h_{lj},
\end{aligned}\\
\begin{aligned}
\frac{\partial h^i_j}{\partial t}=&\frac{\partial g^{ik}h_{kj}}{\partial t}\\
=&2g^{il}h_{lm}g^{mk}h_{kj}+g^{ik}R_{0k0j}-g^{ik}h_{kl}g^{lm}h_{mj}\\
=&g^{ik}R_{0k0j}+g^{il}h_{lm}g^{mk}h_{kj}.
\end{aligned}
\end{gather*}
Define a $(1,1)$-tenser field $Q$ via $Q^i_j=g^{ik}R_{0k0j}$, which is called the Jacobi operator,
then we have the Riccati equation
\begin{equation}\label{Ricatti}
\frac{dA}{dt}=Q+A^2,
\end{equation}
where $A$ is the shape operator. If moreover $N$ is a locally symmetric space, i.e., $\nabla R=0$ on $N$, then we compute
\begin{gather*}
\frac{\partial R_{0i0j}}{\partial t}=\frac{\partial}{\partial t}
R(\frac{\partial}{\partial t},\frac{\partial}{\partial x^i},\frac{\partial}{\partial t},\frac{\partial}{\partial x^j})
=-h_{ik}g^{kl}R_{0l0j}-h_{jk}g^{kl}R_{0i0l},\\
\begin{aligned}
\frac{\partial Q^i_j}{\partial t}=&\frac{\partial g^{im}R_{0m0j}}{\partial t}\\
=&2g^{ik}h_{kl}g^{lm}R_{0l0j}-g^{im}h_{mk}g^{kl}R_{0l0j}-h_{jk}g^{kl}R_{0m0l}g^{im}\\
=&g^{im}h_{mk}g^{kl}R_{0l0j}-h_{jk}g^{kl}R_{0m0l}g^{im}\\
=&h^i_kQ^k_j-Q^i_lh^l_j,
\end{aligned}
\end{gather*}
more concisely,
\begin{equation}\label{curvatureadapted}
\frac{dQ}{dt}=[A,Q].
\end{equation}
We now have the ODE system (\ref{Ricatti})-(\ref{curvatureadapted}) of $A$ and $Q$,
which is independent of other variants.

Riccati equation is also studied in \cite{tubes,ge2015filtration} in other ways.

\subsection{Cartan-M\"{u}nzner polynomial}

In this subsection, we introduce several technical lemmas about Cartan-M\"{u}nzner polynomials.

A polynomial satisfying the Cartan-M\"{u}nzner differential equations
\begin{equation}\label{cartaneq}
\left\{\begin{aligned}
&|\mathrm{grad}\,F|^2=g^2|x|^{2g-2},\\
&\Delta\,F=c|x|^{g-2}
\end{aligned}\right.
\end{equation}
is called a Cartan-M\"{u}nzner polynomial.

By Theorem 3.32 in \cite{thebook} and its constructive proof, we have the following lemma.

\begin{lemma}[see Theorem 3.32 in \cite{thebook}]\label{CM}
Let $\mathcal{F}$ be an isoparametric foliation of $S^n$ with singular leaves $M_+$ and $M_-$,
$g$ the integer determined by $\mathrm{dist}(M_+,M_-)=\frac{\pi}{g}$,
$M_t$ the leaf of $\mathcal{F}$ with $\mathrm{dist}(M_t,M_+)=\frac{\pi}{2g}-t$ ,
or equivalently $\mathrm{dist}(M_t,M_-)=\frac{\pi}{2g}+t$, which is unique.
Let $F_0$ be a function on $S^n$ taking $\sin\frac{\pi t}{2}$ on $M_t$, then $F_0$ is a well-defined smooth function
and $F(x)=|x|^gF_0(\frac{x}{|x|})$ is the unique polynomial taking constant value on each leaf of $\mathcal{F}$,
and satisfying the Cartan-M\"{u}nzner equations with parameters $g$ and $c=g^2(m_--m_+)/2$.
\end{lemma}

We say $F$ is the Cartan-M\"{u}nzner polynoimal of $\mathcal{F}$, or Cartan-M\"{u}nzner polynoimal of $M_t$.

\subsection{Representations of Clifford algebras}\label{repofclifford}

In this subsection, we introduce the Clifford algebras $Cl_m$ and $Cl_m^*$,
and give a representation of each $Cl_m$ and a representation of each $Cl_m^*$.
Representations in this section are real representations. 
Results in this subsection can be found in \cite{spingeometry}.

For convenience, we define
$$I:=\begin{bmatrix}1&\\&1\end{bmatrix},\quad E:=\begin{bmatrix}1&\\&-1\end{bmatrix},\quad
F:=\begin{bmatrix}&1\\1&\end{bmatrix},\quad G:=\begin{bmatrix}&1\\-1&\end{bmatrix}.$$

The Kronecker product of matrices is given in the manner of
$$\begin{bmatrix}a_{11}&\cdots&a_{1n}\\ \vdots&\vdots&\vdots\\a_{11}&\cdots&a_{1n}\end{bmatrix}\otimes A
=\begin{bmatrix}a_{11}A&\cdots&a_{1n}A\\ \vdots&\vdots&\vdots\\a_{11}A&\cdots&a_{1n}A\end{bmatrix}.$$

A Clifford algebra can be defined by $(\bigoplus_{r=0}^\infty\oplus^r\mathbb{R}^m)/\sim$,
where $\sim$ is an equivalence relation given by $v\cdot v=-q(v)1$, with $q$ a quadratic form defined on $\mathrm{R}^m$.
In particular, when $q$ is positive-definite (resp. negative-definite),
in which case we can write $q(v)=|v|^2$ (resp. $q(v)=-|v|^2$),
the Clifford algebra is denoted by $Cl_m$ (resp. $Cl_m^*$).

Up to equivalence, $Cl_m$ (resp. $Cl_m^*$) has one irreducible representation when $m\nequiv3,7\mod8$ (resp. $m\nequiv1,5\mod8$),
two irreducible representations of the same dimension when $m\equiv3,7\mod8$ (resp. $m\equiv1,5\mod8$).
Let $\delta(m)$ and $\delta^*(m)$ denote the dimension of an irreducible representation of $Cl_m$ and $Cl_m^*$, respectively.
Then we have tables
\begin{table}[H]
\centering
\begin{tabular}{|c|c|c|c|c|c|c|c|c|c|}
  \hline
  $m$&$1$&$2$&$3$&$4$&$5$&$6$&$7$&$8$&$m+8$\\
  \hline
  $\delta(m)$&$2$&$4$&$4$&$8$&$8$&$8$&$8$&$16$&$16\delta(m)$\\
  \hline
  \hline
  $m$&$1$&$2$&$3$&$4$&$5$&$6$&$7$&$8$&$m+8$\\
  \hline
  $\delta^*(m)$&$1$&$2$&$4$&$8$&$8$&$16$&$16$&$16$&$16\delta^*(m)$\\
  \hline
\end{tabular}
\end{table}

There are isomorphisms of algebras,
\begin{gather*}
\begin{aligned}
Cl_{m+2}^*\cong\,& Cl_2^*\otimes Cl_m,\\
e_1\leftrightarrow& e_1\otimes1,\\
e_2\leftrightarrow& e_2\otimes1,\\
e_{i+2}\leftrightarrow& e_1e_2\otimes e_i,
\end{aligned}
\quad\quad
\begin{aligned}
Cl_{m+2}\cong\,& Cl_m^*\otimes Cl_2,\\
e_1\leftrightarrow&1\otimes e_1,\\
e_2\leftrightarrow&1\otimes e_2,\\
e_{i+2}\leftrightarrow& e_i\otimes e_1e_2.
\end{aligned}
\end{gather*}
where $(e_1,\dots,e_m)$ is an orthonormal basis of $\mathbb{R}^{m}$ with respect to $q$.
With those isomorphisms and irreducible representations of $Cl_1$, $Cl_2$, $Cl_1^*$ and $Cl_2^*$ given in advance,
we can inductively obtain irreducible representations $\phi_m$ of $Cl_m$
and irreducible representations $\psi_m$ of $Cl_m^*$ for any positive integer $m$.
The constructions is given in the following tables.
\begin{table}[H]
\centering
\caption{Representations of $Cl_m$}
\label{tableeg1}
\begin{tabular}{|c|l|}
  \hline
  Algebra & Representation \\
  \hline
  $Cl_1$ & $e_1\mapsto G$ \\
  \hline
  $Cl_2$ & $e_1\mapsto E\otimes G$, $e_2\mapsto F\otimes G$ \\
  \hline
  $Cl_3$ & $e_1\mapsto E\otimes G$, $e_2\mapsto F\otimes G$, \\
         & $e_3\mapsto G\otimes I$ \\
  \hline
  $Cl_4$ & $e_1\mapsto I\otimes E\otimes G$, $e_2\mapsto I\otimes F\otimes G$, \\
         & $e_3\mapsto E\otimes G\otimes I$, $e_4\mapsto F\otimes G\otimes I$ \\
  \hline
  $Cl_5$ & $e_1\mapsto I\otimes E\otimes G$, $e_2\mapsto I\otimes F\otimes G$, \\
         & $e_3\mapsto E\otimes G\otimes I$, $e_4\mapsto F\otimes G\otimes I$, \\
         & $e_5\mapsto G\otimes G\otimes G$ \\
  \hline
  $Cl_6$ & $e_1\mapsto I\otimes E\otimes G$, $e_2\mapsto I\otimes F\otimes G$, \\
         & $e_3\mapsto E\otimes G\otimes I$, $e_4\mapsto F\otimes G\otimes I$, \\
         & $e_5\mapsto G\otimes I\otimes E$, $e_6\mapsto G\otimes I\otimes F$ \\
  \hline
  $Cl_7$ & $e_1\mapsto I\otimes E\otimes G$, $e_2\mapsto I\otimes F\otimes G$, \\
         & $e_3\mapsto E\otimes G\otimes I$, $e_4\mapsto F\otimes G\otimes I$, \\
         & $e_5\mapsto G\otimes I\otimes E$, $e_6\mapsto G\otimes I\otimes F$, \\
         & $e_7\mapsto G\otimes G\otimes G$ \\
  \hline
\end{tabular}
\end{table}
\begin{table}[H]
\centering
\begin{tabular}{|c|l|}
  \hline
  $Cl_8$ & $e_1\mapsto I\otimes I\otimes E\otimes G$, $e_2\mapsto I\otimes I\otimes F\otimes G$, \\
         & $e_3\mapsto E\otimes I\otimes G\otimes I$, $e_4\mapsto F\otimes I\otimes G\otimes I$, \\
         & $e_5\mapsto G\otimes I\otimes I\otimes E$, $e_6\mapsto G\otimes I\otimes I\otimes F$, \\
         & $e_7\mapsto G\otimes E\otimes G\otimes G$, $e_8\mapsto G\otimes F\otimes G\otimes G$ \\
  \hline
  $Cl_{m+8}$ & $e_i\mapsto\psi_8(e_i)\otimes I_{\delta(m)}$, $1\le i\le8$,\\
             & $e_j\mapsto G\otimes G\otimes G\otimes G\otimes\psi_m(e_{j-8})$, $9\le j\le m+8$\\
  \hline
\end{tabular}
\end{table}
\begin{table}[H]
\centering
\caption{Representations of $Cl_m^*$}
\label{tableeg2}
\begin{tabular}{|c|l|}
  \hline
  Algebra & Representation \\
  \hline
  $Cl_1^*$ & $e_1\mapsto 1$ \\
  \hline
  $Cl_2^*$ & $e_1\mapsto E$, $e_2\mapsto F$ \\
  \hline
  $Cl_3^*$ & $e_1\mapsto E\otimes I$, $e_2\mapsto F\otimes I$, \\
           & $e_3\mapsto G\otimes G$ \\
  \hline
  $Cl_4^*$ & $e_1\mapsto E\otimes I\otimes I$, $e_2\mapsto F\otimes I\otimes I$, \\
           & $e_3\mapsto G\otimes E\otimes G$, $e_4\mapsto G\otimes F\otimes G$ \\
  \hline
  $Cl_5^*$ & $e_1\mapsto E\otimes I\otimes I$, $e_2\mapsto F\otimes I\otimes I$, \\
           & $e_3\mapsto G\otimes E\otimes G$, $e_4\mapsto G\otimes F\otimes G$, \\
           & $e_5\mapsto G\otimes G\otimes I$ \\
  \hline
  $Cl_6^*$ & $e_1\mapsto E\otimes I\otimes I\otimes I$, $e_2\mapsto F\otimes I\otimes I\otimes I$, \\
           & $e_3\mapsto G\otimes I\otimes E\otimes G$, $e_4\mapsto G\otimes I\otimes F\otimes G$, \\
           & $e_5\mapsto G\otimes E\otimes G\otimes I$, $e_6\mapsto G\otimes F\otimes G\otimes I$ \\
  \hline
  $Cl_7^*$ & $e_1\mapsto E\otimes I\otimes I\otimes I$, $e_2\mapsto F\otimes I\otimes I\otimes I$, \\
           & $e_3\mapsto G\otimes I\otimes E\otimes G$, $e_4\mapsto G\otimes I\otimes F\otimes G$, \\
           & $e_5\mapsto G\otimes E\otimes G\otimes I$, $e_6\mapsto G\otimes F\otimes G\otimes I$, \\
           & $e_7\mapsto G\otimes G\otimes G\otimes G$ \\
  \hline
  $Cl_8^*$ & $e_1\mapsto E\otimes I\otimes I\otimes I$, $e_2\mapsto F\otimes I\otimes I\otimes I$, \\
           & $e_3\mapsto G\otimes I\otimes E\otimes G$, $e_4\mapsto G\otimes I\otimes F\otimes G$, \\
           & $e_5\mapsto G\otimes E\otimes G\otimes I$, $e_6\mapsto G\otimes F\otimes G\otimes I$, \\
           & $e_7\mapsto G\otimes G\otimes I\otimes E$, $e_8\mapsto G\otimes G\otimes I\otimes F$ \\
  \hline
  $Cl_{m+8}$ & $e_i\mapsto\psi_m(e_i)\otimes G\otimes G\otimes G\otimes G$, $1\le i\le m$,\\
             & $e_j\mapsto I_{\delta^*(m)}\otimes\psi_8(e_{j-m})$, $m+1\le j\le m+8$\\
  \hline
\end{tabular}
\end{table}

The product of two complex numbers, two quaternions or two octonions can be represented by Clifford algebras.
More precisely,
$$(x^0-x^1\mathrm{i})(y^0+y^1\mathrm{i})=(x^0y^0+x^1y^1)+(x^0y^1-x^1y^0)\mathrm{i}
=\langle x,y\rangle^2+\langle x,E_1y\rangle^2,$$
where $E_1=\phi_1(e_1)$ in Table \ref{tableeg1};
$$\begin{aligned}
&(x^0-x^1\mathrm{i}-x^2\mathrm{j}-x^3\mathrm{k})(y^0+y^1\mathrm{i}+y^2\mathrm{j}+y^3\mathrm{k})\\
=&(x^0y^0+x^1y^1+x^2y^2+x^3y^3)+(x^0y^1-x^1y^0-x^2y^3+x^3y^2)\mathrm{i}\\
&+(x^0y^2+x^1y^3-x^2y^0-x^3y^1)\mathrm{j}+(x^0y^3-x^1y^2+x^2y^1-x^3y^0)\mathrm{k}\\
=&\langle x,y\rangle^2+\langle x,E_1y\rangle^2+\langle x,E_3y\rangle^2+\langle x,E_2y\rangle^2,
\end{aligned}$$
where $E_i=\phi_3(e_i)$ in Table \ref{tableeg1}, $1\le i\le3$;
$$\begin{aligned}
&(x^0-\sum_{i=1}^7x^ie_i)(y^0+\sum_{i=1}^7y^ie_i)\\
=&(x^0y^0+x^1y^1+x^2y^2+x^3y^3+x^4y^4+x^5y^5+x^6y^6+x^7y^7)\\
&+(x^0y^1-x^1y^0-x^2y^3+x^3y^2+x^4y^5-x^5y^4-x^6y^7+x^7y^6)e_1\\
&+(x^0y^2+x^1y^3-x^2y^0-x^3y^1-x^4y^6-x^5y^7+x^6y^4+x^7y^5)e_2\\
&+(x^0y^3-x^1y^2+x^2y^1-x^3y^0+x^4y^7-x^5y^6+x^6y^5-x^7y^4)e_3\\
&+(x^0y^4-x^1y^5+x^2y^6-x^3y^7-x^4y^0+x^5y^1-x^6y^2+x^7y^3)e_4\\
&+(x^0y^5+x^1y^4+x^2y^7+x^3y^6-x^4y^1-x^5y^0-x^6y^3-x^7y^2)e_5\\
&+(x^0y^6+x^1y^7-x^2y^4-x^3y^5+x^4y^2+x^5y^3-x^6y^0-x^7y^1)e_6\\
&+(x^0y^7-x^1y^6-x^2y^5+x^3y^4-x^4y^3+x^5y^2+x^6y^1-x^7y^0)e_7\\
=&\langle x,y\rangle^2+\langle x,E_1y\rangle^2+\langle x,E_3y\rangle^2+\langle x,E_2y\rangle^2\\
&+\langle x,E_5y\rangle^2+\langle x,E_6y\rangle^2+\langle x,E_4y\rangle^2+\langle x,E_7y\rangle^2,
\end{aligned}$$
where $E_i=\phi_7(e_i)$ in Table \ref{tableeg1}, $1\le i\le7$, and the product of octonions follows the laws
\begin{itemize}
\item $(e_i)^2=-1$, for $1\le i\le7$;
\item if $e_ie_j=-e_je_i$, for $1\le i,j\le7$ with $i\ne j$;
\item if $e_ie_j=e_k$, then $e_je_k=e_i$ and $e_ke_i=e_j$, for $1\le i,j\le7$ with $i\ne j$;
\item $e_1e_2=e_3$, $e_1e_5=e_4$, $e_1e_6=e_7$, $e_2e_4=e_6$,
$e_2e_5=e_7$, $e_4e_3=e_7$, $e_6e_3=e_5$. 
\end{itemize}

\section{Examples of Isoparametric Functions in $S^n\times S^n$}\label{eg}

\subsection{Examples of bi-degree $(1,1)$}\label{egr}

Let $F_r:\mathbb{R}^{n+1}\times\mathbb{R}^{n+1}\to\mathbb{R},\,F_r(x,y)=\langle x,y\rangle=\sum_{i=0}^nx^iy^i$,
where $n\ge2$, then by direct computation we obtain
\begin{gather*}
\mathrm{grad}\,F_r=(y^0,\dots,y^n,x^0,\dots,x^n),\\
|\mathrm{grad}\,F_r|^2=|x|^2+|y|^2,\\
\Delta\,F_r=0.
\end{gather*}
Therefore, $F_r|_{S^n\times S^n}$ is isoparametric.

This family of isoparametric functions on $S^n\times S^n$ is first proposed by Qian and Tang \cite{qian2016isoparametric},
and Urbano \cite{urbano2016S2S2}.

\subsection{Examples arisen from products of $\mathbb{C}$,$\mathbb{H}$,$\mathbb{O}$}\label{egcho}

Let $F_c:\mathbb{C}^{n+1}\times\mathbb{C}^{n+1}\to\mathbb{R}$, $F_c(z,w)=|\sum_{i=0}^n\overline{z^i}w^i|^2$, where $n\ge2$. We compute
$$\frac{\partial F_c}{\partial z^i}=\overline{w^i}(\sum_{i=0}^n\overline{z^i}w^i),\quad
\frac{\partial F_c}{\partial w^i}=\overline{z^i}(\sum_{i=0}^nz^i\overline{w^i}).$$
Since $|\frac{\partial F_c}{\partial z^i}|=|\frac{\partial\overline{F_c}}{\partial z^i}|
=|\overline{(\frac{\partial F_c}{\partial\overline{z^i}})}|=|\frac{\partial F_c}{\partial\overline{z^i}}|$,
and similarly $|\frac{\partial F_c}{\partial w^i}|=|\frac{\partial F_c}{\partial\overline{w^i}}|$,
$$\begin{aligned}|\mathrm{grad}\,F_c|^2
=&2\sum_{i=0}^n(|\frac{\partial F_c}{\partial z^i}|^2+|\frac{\partial F_c}{\partial\overline{z^i}}|^2
+|\frac{\partial F_c}{\partial w^i}|^2+|\frac{\partial F_c}{\partial\overline{w^i}}|^2)\\
=&4\sum_{i=0}^n(|\frac{\partial F_c}{\partial z^i}|^2+|\frac{\partial F_c}{\partial w^i}|^2)\\
=&4\sum_{i=0}^n(|w^i|^2F_c+|z^i|^2F_c)\\
=&4(|z^2|+|w|^2)F_c.
\end{aligned}$$
We also compute,
$$\Delta\,F_c=4\sum_{i=0}^n(\frac{\partial^2F_c}{\partial z^i\partial\overline{z^i}}+\frac{\partial^2F_c}{\partial w^i\partial\overline{w^i}})
=4(|z|^2+|w|^2).$$
Therefore, $F_c|_{S^{2n+1}\times S^{2n+1}}$ is isoparametric.

We also have similar examples in quaternion and octonion cases,
\begin{gather*}
F_h:\mathbb{H}^{n+1}\times\mathbb{H}^{n+1}\to\mathbb{R},\quad F_h(p,q)=|\sum_{l=0}^n(p^l)^*q^l|^2,\\
F_o:\mathbb{O}^{n+1}\times\mathbb{O}^{n+1}\to\mathbb{R},\quad F_o(P,Q)=|\sum_{l=0}^n(P^l)^*Q^l|^2,
\end{gather*}
where $n\ge2$, which can be verified to be isoparametric for $S^{4n+3}\times S^{4n+3}$ and $S^{8n+7}\times S^{8n+7}$, respectively,
in the same way as above.

\subsection{Clifford-type examples}\label{exampleclifford}

Representations in this section are real representations. 

Let $(e_1,\dots,e_{m-1})$ be an orthonormal basis of $\mathbb{R}^{m-1}$, $\phi$ a representation of $Cl_{m-1}$,
$E_i=\phi(e_i)$, $i=1,\dots,m-1$. We define
\begin{equation*}
F(x,y)=\langle x,y\rangle^2+\sum_{i=1}^{m-1}\langle x,E_iy\rangle^2,
\end{equation*}
which is labeled by (\ref{cuieg}) in the Theorem.
We compute, 
\begin{gather*}
\mathrm{grad}\,F=2\langle x,y\rangle(y,x)
+\sum_{i=1}^{m-1}(2\langle x,E_iy\rangle(E_iy,0)+2\langle E_ix,y\rangle(0,E_ix)),\\
|\mathrm{grad}\,F|^2=4\langle x,y\rangle^2(|x|^2+|y|^2)+\sum_{i=1}^{m-1}4\langle x,E_iy\rangle^2(|x|^2+|y|^2),\\
\frac{\partial^2 F}{(\partial x^j)^2}=2y^jy^j+2\sum_{i=1}^{m-1}(E_iy)^j(E_iy)^j,\\
\Delta^xF=2m|y|^2,\quad\Delta^yF=2m|x|^2,\\
\Delta\,F=2m(|x|^2+|y|^2),
\end{gather*}
where in the computation of the second equation,
we have used the fact that $\langle E_ix,E_j x\rangle=0$ when $i\ne j$.
Therefore, $F|_{S^n\times S^n}$ is isoparametric.

When $m\equiv3,7\mod8$, $Cl_m$ has two irreducible representations $\phi_m$ and $\phi_m'$ up to equivalence,
and $\phi_m'$ can be given by $\phi_m'(e_i)=-\phi_m(e_i)$.
Since $\langle x,\phi(e_i)y\rangle^2=\langle x,-\phi(e_i)y\rangle^2$, $F$ defined by $k\phi_m+l\phi_m'$
equals $F$ defined by $(k+l)\phi_m$. Therefore, for each $(m,k)$, there is one and only one function
of the form (\ref{cuieg}) defined by  a representation of $Cl_{m-1}$ with $k$ irreducible summands.

Using the Clifford representation of products of complex numbers, quaternions, and octonions in Subsection \ref{repofclifford},
we find that
\begin{gather*}
F_c(x,y)=\langle x,y\rangle^2+\langle x,E_1y\rangle^2,\\
F_q(x,y)=\langle x,y\rangle^2+\sum_{i=1}^3\langle x,E_iy\rangle^2,\\
F_o(x,y)=\langle x,y\rangle^2+\sum_{i=1}^7\langle x,E_iy\rangle^2,
\end{gather*}
where $E_i=\phi(e_i)$, $\phi$ is a reducible representation of $Cl_1$,$Cl_3$ and $Cl_7$, respectively.
This corresponding is suggested by Cui after his discovery of formula (\ref{cuieg}).

There are also equivalence of (3) and other isoparametric functions on $S^n\times S^n$.
When $(m,n)=(2,1),(4,3)$ or $(8,7)$,
$F(x,y)=|x\cdot y|^2=|x|^2|y|^2,$
trivial on $S^n\times S^n$.
When $m=1$,
$F(x,y)=\langle x,y\rangle^2,$
determine the same foliation as $F_r$ does.
When $(m,n)=(3,3)$ or $(7,7)$,
$F(x,y)=|x\cdot y|^2-\langle x,y\rangle^2=|x|^2|y|^2-\langle x,y\rangle^2,$
also determine the same foliation as $F_r$ does.

There is another equivalent expression of this family of examples, which involves representations of $Cl_m^*$.
\begin{gather}
F^*(x,y)=\sum_{i=1}^m\langle P_ix,y\rangle^2,\text{ when }m\equiv1,2,3,5\mod 8, \label{fulleg}\\
F^*(x,y)=\sum_{i=1}^m(X^tP_iX)^2,\text{ when }m\equiv4,6,7,8\mod 8,\label{notfulleg}
\end{gather}
where $\psi:e_i\mapsto P_i$, $1\le i\le m$, is a representation of $Cl_m^*$,
$X=\begin{bmatrix}x\\y\end{bmatrix}\in\mathbb{R}^{n+1}\times\mathbb{R}^{n+1}$,
and $P_i$ in (\ref{notfulleg}) are selected to be QSD as in the proof of Lemma \ref{QSD}.

The examples (\ref{fulleg}) can be rewritten in the form
$$F^*(x,y)=\sum_{i=1}^m(X^t\begin{bmatrix}&P_i\\P_i&\end{bmatrix}X)^2,$$
which is congruent to
$$F^*(x,y)=\sum_{i=1}^m(X^t\begin{bmatrix}P_i&\\&-P_i\end{bmatrix}X)^2,$$
with respect to $S^{2n+1}$.

When $m\equiv2,3\mod8$, $Cl_m^*$ has one irreducible representation $\psi_m$ up to equivalence,
and the representation given by $e_i\mapsto\psi_m(e_i)$ and $e_i\mapsto-\psi_m(e_i)$ are equivalent.
If $\psi\simeq k\psi_m$, then $F^*$ is congruent to
\begin{equation}\label{FKMexpress1}
F':=\sum_{i=1}^m(X^t(I_{2k}\otimes\psi_m(e_i))X)^2,
\end{equation}
which is an FKM-type isoparametric polynomial for $S^{2n+1}$.

When $m\equiv1,5\mod8$, $Cl_m^*$ has two irreducible representations $\psi_m$ and $\psi_m'$ up to equivalence,
and $\psi_m'$ can be given by $\psi_m'(e_i)=-\psi_m(e_i)$.
If $\psi\sim k\psi_m+l\psi_m'$, then $F^*$ is congruent to
$$\sum_{i=1}^m(X^t\mathrm{diag}(I_k\otimes\psi_m(e_i),-I_l\otimes\psi_m(e_i),-I_k\otimes\psi_m(e_i),I_l\otimes\psi_m(e_i))X)^2,$$
and then congruent to
\begin{equation}\label{FKMexpress2}
\sum_{i=1}^m(X^t\begin{bmatrix}I_{k+l}\otimes\psi_m(e_i)&\\&I_{k+l}\otimes\psi'_m(e_i)\end{bmatrix}X)^2=:F',
\end{equation}
which is an FKM-type isoparametric polynomial for $S^{2n+1}$.

Therefore, for each $(m,k)$ there is one and only one function
of the form (\ref{fulleg}) or (\ref{notfulleg}) defined by  a representation of $Cl_m^*$ with $k$ irreducible summands.

The FKM-type isoparametric polynomial $F'$ in (\ref{FKMexpress1}) and (\ref{FKMexpress2}) satisfies
$$|\mathrm{grad}\,F'|^2=4(|x|^2+|y|^2)F',\quad\Delta\,F'=c(|x|^2+|y|^2).$$
Hence $F^*$ in (\ref{fulleg}) and (\ref{notfulleg}) also satisfies
$$|\mathrm{grad}\,F^*|^2=4(|x|^2+|y|^2)F^*,\quad\Delta\,F^*=c(|x|^2+|y|^2).$$
Therefore, the two families of polynomials (\ref{fulleg}) and (\ref{notfulleg}) are isoparametric for $S^n\times S^n$.

\subsection{Examples on $S^7\times S^7$}

Let $M:S^7\times S^7\to S^7$, $(x,y)\to x^*y$ be given by the multiplication of octonions.
If $f:S^7\to \mathbb{R}$ is an isomoparametric function on $S^7$, then $f\circ M$ is an isomoparametric function on $S^7\times S^7$.
This follows from the following lemma.

\begin{lemma}\label{Riesubm}
Let $\pi:M\to N$ be a fibre bundle with $M$ and $N$ endowed with metrics making $\pi$ a Riemannian submersion
and every fibre a minimal submanifold of $M$. Let $F$ be a function on $N$,
then $F$ is an isoparametric function on $N$ if and only if $\pi^*F$ is an isoparametric funtinon on $M$.
\end{lemma}

\begin{proof}
We prove the lemma by proving that for any $p\in M$,
$$|\nabla^MF|^2|_p=|\nabla^NF|^2|_{\pi(p)},\quad\Delta^MF|_p=\Delta^NF|_{\pi(p)},$$
where by abuse of notation we denote $\pi^*F$ by $F$.
Let $\mathscr{K}$ be the tangent distribution of the fibres.
Let $(X_1,\dots,X_n)$ be a local orthonormal frame of $M$ in a neighbourhood of $p$,
such that $X_1,\dots,X_k\in\mathscr{K}$ and that $\pi_*(X_{k+1}),\dots,\pi_*(X_n)$ is well-defined.
Then $\pi_*(X_{k+1}),\dots,\pi_*(X_n)$ is a local orthogonal frame of $N$ in a neighbourhood of $\pi(p)$.
Let $\nabla_X^\top Y$ and $\nabla_X^\bot Y$ be the projections of $\nabla^M_XY$ to $\mathscr{K}$ and its
normal distribution, respectively.
We compute,
$$
|\nabla^NF|^2=\sum_{\alpha=k+1}^n|(\pi_*X_\alpha)F|^2=\sum_{\alpha=k+1}^n|X_\alpha F|^2
=\sum_{A=1}^n|X_A F|^2=|\nabla^MF|^2,
$$
where the third equality is because $F$ takes constant value on each fibre.
At $p$ we compute,
\begin{gather*}
\begin{aligned}
\Delta^NF=&\sum_{\alpha=k+1}^n((\pi_*X_\alpha)((\pi_*X_\alpha)F)-(\nabla^M_{\pi_*X_\alpha}\pi_*X_\alpha)F)\\
=&\sum_{\alpha=k+1}^n(X_\alpha(X_\alpha F)-(\nabla_{X_\alpha}^\bot X_\alpha)F),
\end{aligned}\\
\begin{aligned}
\Delta^MF=&\sum_{\alpha=1}^k(X_i(X_i F)-(\nabla_{X_i}^M X_i)F)
+\sum_{\alpha=k+1}^n(X_\alpha(X_\alpha F)-(\nabla_{X_\alpha}^M X_\alpha)F)\\
=&-\sum_{\alpha=1}^k(\nabla_{X_i}^\bot X_i)F
+\sum_{\alpha=k+1}^n(X_\alpha(X_\alpha F)-(\nabla_{X_\alpha}^\bot X_\alpha)F),
\end{aligned}\\
\Delta^NF-\Delta^MF=HF=0,
\end{gather*}
where $H$ is the mean curvature vector,
and we have used the fact
$\pi_*(\nabla_{X_\alpha}^\bot X_\alpha)=\nabla_{\pi_*{X_\alpha}}^N \pi_*{X_\alpha}$,
which can be obtained from basic properties of Riemannian submersion.
\end{proof}

Let $K\le H\le G$ be a chain of Lie subgroups, which induces a principal bundle $H/K\to G/K\to G/H$,
then a function on $G/H$ is isoparametric if and only if its pullback to $G/K$ under the bundle projection is isoparametric, by Lemma \ref{Riesubm}.

\section{Properties of Isoparametric Hypersurfaces in $S^n\times S^n$}\label{property}

In order to understand isoparametric hypersurfaces in $S^n\times S^n$ better,
we inquire some fundamental properties of them,
analogous to those classical properties of isoparametric hypersurfaces in $S^n$.

\subsection{Isoparametric polynomials for $S^n\times S^n$}

For clarity, we restate (A) of the Theorem as the following lemma.

\begin{lemma}\label{isoparabebihomo}
Let $M$ be a connected closed isoparametric hypersurface in $S^n\times S^n$, i.e., a regular level set of
an isoparametric function $F_0$ defined on a neighbourhood of $M$.
Then there exists a bi-homogeneous function $F$, say, of bi-degree $(g,g)$,
on $\mathbb{R}^{n+1}\times\mathbb{R}^{n+1}$ such that
\begin{itemize}
\item[1.]$F$ takes constant values on each regular level set of $F_0$,
$\mathrm{im}\,F|_{S^n\times S^n}=[-1,1]$;
\item[2.]for any $x\in S^n$ (resp. $y\in S^n$),
$\mathrm{im}\,F|_{\{x\}\times S^n}$ (resp. $\mathrm{im}\,F|_{S^n\times\{y\}}$) equals $[-1,1]$;
\item[3.]for any $x\in S^n$ (resp. $y\in S^n$), $F|_{\{x\}\times S^n}$ (resp. $F|_{S^n\times\{y\}}$)
is the Cartan-M{\"u}nzner polynomial of $M\cap(\{x\}\times S^n)$ (resp. $M\cap(S^n\times\{y\})$),
for any $y\in S^n$;
\item[4.]$g=1,2,3,4,\text{or }6$.
\end{itemize}
\end{lemma}

\begin{proof}
Let $E:M\times\mathbb{R}\to S^n\times S^n$ be the normal exponential map.
Since $M$ is closed, $E$ is epimorphic.
For any $p,q\in M$, $H(p,t)=H(q,t)$ for $t\in (-\varepsilon,\varepsilon)$, $\varepsilon$ small enough.
Since the solutions to (\ref{Ricatti})-(\ref{curvatureadapted}) are analytic functions of $t$,
for each $p\in M$, $H(p,t)$ can be extended to a meromorphic function and $H(p,t)=H(q,t)$ for any $p,q\in M$ and $t\in \mathbb{R}$. 
Therefore, $H$ is a constant function on each $M_t$, whenever it is finite.
When $H(p,t)$ is finite, $dE|_{(p,t)}$ is non-degenerate. Hence $E(M,t)$ is a hypersurface when $H(t)$ is finite.

Since $M$ is a closed connected hypersurface of an oriented manifold,
$M$ is orientable and has two sides. Let $t_1$ and $t_2$ be the greatest negative sigular time
and the smallest positive singular time of the normal exponential map, respectively.
By the variation theory of geodesic lines, $E(M,[t_1,t_2])=S^n\times S^n$, $E|_{M\times(-1,1)}$ is monomorphic,
and the function $F:S^n\times S^n\to\mathbb{R}$, $E(p,t)\mapsto \sin\frac{\pi}{t_2-t_1}(t-\frac{t_1+T_2}{2})$ is well-defined.
Then $\mathrm{im}(F)=[-1,1]$.

Let $M_t:=F^{-1}(sin\frac{\pi t}{2})$.
Then $M_t$ is a connnected closed hypersurface in $S^n\times S^n$, when $t\in(-1,1)$, $M_+:=M_1$ and $M_-:=M_{-1}$
are its focal varieties.
By a) of Theorem A in \cite{wang1987isoparametric},
which states that the focal variety of a transnormal hypersurface is a smooth submanifold,
$M_+$ and $M_-$ are submanifolds of codimesion $>1$.

Let $\pi_1$ and $\pi_2$ be the projection of $S^h\times S^n$ to the first and the second entries, respectively. 
Assume the normal vector to a regular level set $M_t$ at $(x,y)$ is $(\varepsilon_0,\varepsilon_0')$,
extend $(\varepsilon_0)$ to a basis $(\varepsilon_0,\varepsilon_1,\dots,\varepsilon_{n-1})$ of $T_xS^n$.
Then $({\pi_1}_*(-\frac{\varepsilon_0}{|\varepsilon_0'|^2},\frac{\varepsilon_0'}{|\varepsilon_0|^2}),\\
{\pi_1}_*(\varepsilon_1,0),\dots,{\pi_1}_*(\varepsilon_{n-1},0))$ forms a basis of $T_xS^n$.
Since $\pi_*$ is not degenerate at any point of $M_t$, $-1\le t\le1$, $\pi(M_t)$ is an open subset of $S^n$.
Since $\pi(M_t)$ is also a closed non-empty subset of $S^n$, $\pi_1(M_t)=S^n$.
Assume $M_+=E(M,T)$.
For any $x\in S^n$, there exists $p_k$ such that $\pi_1\circ E(p_k,(1-2^{-k})T)=x$.
There exists a Cauchy subsequence of $\{(p_k,(1-2^{-k})T)\}$ converging to $(p,T)$, and $\pi_1\circ E(p,T)=x$.
Hence $\pi_1(M_+)=S^n$, and similarly $\pi_1(M_-)=S^n$.
Hence for any $t\in[-1,1]$ and $x\in S^n$, $M_t\cap(\{x\}\times S^n)$ is non-empty,
and $[-1,1]\subseteq \mathrm{im}\,F|_{\{x\}\times S^n}$.
Therefore, $F|_{\{x\}\times S^n}=[-1,1]$, and similarly, $F|_{S^n\times\{y\}}=[-1,1]$.

Since the normal distribution of an isoparametric foliation is geodesic (proven in Subsection \ref{ODE}),
on an integral (geodesic) line of $\mathrm{grad}_2\,F$,
there exsit positive numbers $a$ and $b$ with $a^2+b^2=1$ such that
$|\widetilde{\mathrm{grad}}^x\,F|=a|\widetilde{\mathrm{grad}}\,F|$, and that $|\widetilde{\mathrm{grad}}^y\,F|=b|\widetilde{\mathrm{grad}}\,F|$.
The integral line of $\mathrm{grad}\,F$ passing through $p$ is of finite length if and only if
$a/b$ is rational. If $a$ or $b$ is not constant, then $a/b$ ranges over all rational numbers of some interval.
In that case, there exist two values of $a/b$,
whose corresponding integral circles have irrational length ratio. This contradict with the fact that
lengths of all closed integral lines are integer multiples of $\mathrm{dist}(M_+,M_-)$.
Hence $a$ and $b$ are constant functions. For fixed $y\in S^n$, since $|\widetilde{\mathrm{grad}}^x\,F|$ takes
a constant value on each level set of $F|_{S^n\times\{y\}}$,
$\{M_t\cap(S^n\times\{y\})\}$ forms a singular transnormal foliation on $S^n$.
It is also an isoparametric foliation, because of Theorem B in \cite{wang1987isoparametric}, 
which states that transnormal functions on $S^n$ and $\mathbb{R}^n$ are isoparametirc.
Hence $F_0|_{S^n\times\{y\}}$ is isoparametric. For the same reason, $F|_{\{x\}\times S^n}$ is isoparametric.

Let $m_+:=\dim M_0-\dim M_+$, $m_-:=\dim M_0-\dim M_-$.
For any $p=(x,y)\in M_+$,
since $\pi_2(X)=X-\pi_1(X)\in T_pM_++T_p(S^n\times\{y\})$ and $\pi_2(T_pM_+)=T_yS^n$,
we have $T_p(M_+)+T_p(S^n\times\{y\})=T_p(\{x\}\times S^n)+T_p(S^n\times\{y\})=T_p(S^n\times S^n)$.
We compute
\begin{equation*}\begin{aligned}
&\dim(T_pM_+\cap T_p(S^n\times\{y\}))\\
=&\dim T_pM+\dim(S^n\times\{y\})-\dim(T_pM_++T_p(S^n\times\{y\}))\\
=&(2n-1-m_+)+n-2n\\
=&n-1-m_+.
\end{aligned}\end{equation*}
Hence $M_+\cap(S^n\times\{y\})$ is an $(n-1-m_+)$-dimensional submanifold of $S^n\times\{y\}$ for any $y\in S^n$.
Similarly, $M_-\cap(S^n\times\{y\})$,$M_+\cap(S^n\times\{x\})$ and $M_-\cap(S^n\times\{x\})$
are $(n-1-m_-)$-dimensional, $(n-1-m_+)$-dimensional and $(n-1-m_-)$-dimensional submanifolds of
$S^n\times\{y\}$, $S^n\times\{x\}$ and $S^n\times\{x\}$, respectively.
Suppose the isoparametric foliations $\{M_t\cap(S^n\times\{y\})\}$ and $\{M_t\cap(\{x\}\times S^n)\}$ are of degree $r$ and $s$, respectively.
Since $r(m_++m_-)=2n-2=s(m_++m_-)$, we have $r=s$. Let $g:=r$.

The domain of $F$ can be extended to $\mathbb{R}^{n+1}\times\mathbb{R}^{n+1}$ via
$F(x,y)=|x|^g|y|^gF(\frac{x}{|x|},\frac{y}{|y|})$.
Since as a function on $S^n\times\{y\}$, $F(M_t\cap(S^n\times\{y\}))=\sin\frac{\pi t}{2}$. 
Hence, by Lemma \ref{CM}, $F$ coincides with the Cartan-M\"{u}nzner polynomial of $\{M_t\}$, and $x\mapsto F(x,y)$ is a polynomial of degree $g$. 
Similarly, for any $x$, $y\mapsto F(x,y)$ is an isoparametric polynomial of degree $g$ for $S^n$.

Since
\begin{gather*}
\frac{\partial^{2g+1}F}{\partial x^{i_1}\cdots\partial x^{i_{g+1}}\partial y^{j_1}\cdots\partial y^{j_g}}=0
\text{ for any } i_1,\dots,i_{g+1},j_1,\dots,j_g,\\
\frac{\partial^{2g+1} F}{\partial x^{i_1}\cdots\partial x^{i_g}\partial y^{j_1}\partial\cdots y^{j_{g+1}}}=0
\text{ for any } i_1,\dots,i_g,j_1,\dots,j_{g+1},
\end{gather*}
$F$ is a polynomial of bidegree $(g,g)$.
\end{proof}

We say an isoparametric hypersurface or an isoparametric foliation is of bi-degree $(g,g)$,
if the polynomial it determines as in Lemma \ref{isoparabebihomo} is of bi-degree $(g,g)$.

\subsection{The second fundamental form}\label{2ndfundform}

Let $F$ be a isoparametric polynomial function on $\mathbb{R}^{2n+2}$ of bi-degree $(g,g)$
with $\mathrm{im}\,F|_{S^n\times S^n}=[-1,1]$, $M$ a regular level set of $F|_{S^n\times S^n}$. We compute the second fundamental form
of $M\hookrightarrow S^n\times S^n$.

For $(x,y)\in S^n\times S^n$, on a neighbourhood of $(x,y)$ in $S^n\times S^n$, we define an orthonormal frame
\begin{gather*}
e_0=\frac{1}{\sqrt{1-F^2}}(\frac{1}{g}\frac{\partial F}{\partial x}-Fx,0),
e_1=(\varepsilon_1,0),\dots,e_{n-1}=(\varepsilon_{n-1},0),\\
e'_0=\frac{1}{\sqrt{1-F^2}}(0,\frac{1}{g}\frac{\partial F}{\partial y}-Fy),
e'_1=(0,\varepsilon'_1),\dots,e'_{n-1}=(0,\varepsilon'_{n-1}).
\end{gather*}
The projection of $e_0$ (resp. $e'_0$) to $M$ is $\frac{e_0-e'_0}{2}$ (resp. $\frac{e'_0-e_0}{2}$).

We compute,
\begin{gather*}
\widetilde{\nabla}_{e_i}e_j=(\widetilde{\nabla}_{\varepsilon_i}\varepsilon_j,0)
=(\widehat{\nabla}_{\varepsilon_i}\varepsilon_j+B^1(\varepsilon_i,\varepsilon_j)\varepsilon_0,0)
=\sum_{i=1}^{n-1}c_ie_i+B^1(\varepsilon_i,\varepsilon_j)e_0,\\
\widetilde{\nabla}_{e'_i}e'_j=\sum_{i=1}^{n-1}c'_ie'_i+B^2(\varepsilon'_i,\varepsilon'_j)e'_0,\\
\widetilde{\nabla}_{e_i}e_0=(\widetilde{\nabla}_{\varepsilon_i}\varepsilon_0,0)
=(-\sum_{j=0}^{n-1}B^1(\varepsilon_i,\varepsilon_j)\varepsilon_j,0)
=-\sum_{j=0}^{n-1}B^1(\varepsilon_i,\varepsilon_j)e_j,\\
\widetilde{\nabla}_{e'_i}e'_0=-\sum_{j=0}^{n-1}B^2(\varepsilon'_i,\varepsilon'_j)e'_j.
\end{gather*}
Since the integral curves of $e_0$ and $e'_0$ are geodesics lines in $S^{n+1}\times S^{n+1}$, we have
$$\widetilde{\nabla}_{e_0}e_0=0,\quad\widetilde{\nabla}_{e'_0}e'_0=0.$$
Noticing that $e_0=\frac{1}{\sqrt{1-F^2}}\sum_{i=1}^{n-1}(\frac{1}{g}\frac{\partial F}{\partial x^i}-Fx^i)\frac{\partial}{\partial x^i}$,
we compute the $j$'th entry of $\sqrt{1-F^2}\,\nabla_{e_0}\frac{\partial F}{\partial y}$,
$$\begin{aligned}
&\sum_{i=0}^{n-1}(\frac{1}{g}\frac{\partial F}{\partial x^i}-Fx^i)\frac{\partial^2 F}{\partial x^i\partial y^j}
=\sum_{i=0}^{n-1}\frac{1}{2g}\frac{\partial}{\partial y^j}(\frac{\partial F}{\partial x^i})^2
-F\frac{\partial}{\partial y^j}(\sum_{i=0}^nx^i\frac{\partial F}{\partial x^i})\\
=&\frac{1}{2g}\frac{\partial}{\partial y^j}(g^2|x|^{2g-2}|y|^{2g})-F\frac{\partial}{\partial y^j}(gF)
=g^2y^j-gF\frac{\partial F}{\partial y^j}.
\end{aligned}$$
Hence $\sqrt{1-F^2}\,\nabla_{e_0}\frac{\partial F}{\partial y^0}=(0,g^2y-gF\frac{\partial F}{\partial y})$.
Also noting that
$$\sqrt{1-F^2}e_0(\frac{1}{1-F^2})=\frac{gF}{\sqrt{1-F^2}},\quad
\sqrt{1-F^2}e_0(\frac{F}{1-F^2})=\frac{g}{\sqrt{1-F^2}},$$
we compute,
$$\begin{aligned}
\sqrt{1-F^2}\,\nabla_{e_0}e'_0=&\sqrt{1-F^2}\nabla_{e_0}
(0,\frac{1}{g\sqrt{1-F^2}}\frac{\partial F}{\partial y}-\frac{F}{\sqrt{1-F^2}}y)\\
=&(0,\frac{F}{\sqrt{1-F^2}}\frac{\partial F}{\partial y}
+\frac{1}{\sqrt{1-F^2}}(gF-F\frac{\partial F}{\partial y})-\frac{g}{\sqrt{1-F^2}}y)\\
=&0.
\end{aligned}$$
Therefore, $\widetilde{\nabla}_{e_0}e'_0=0$. For the same reason, $\widetilde{\nabla}_{e'_0}e_0=0$.

By the above computation, the second fundamental form $B$ of $M\hookrightarrow S^n\times S^n$ satisfies
\begin{gather*}
\left\{\begin{aligned}
B(e_i,e_j)=&\langle\widetilde{\nabla}_{e_i}e_j,\frac{e_0+e'_0}{\sqrt{2}}\rangle
=B^1(\varepsilon_i,\varepsilon_j)\langle e_0,\frac{e_0+e'_0}{\sqrt{2}}\rangle
=\frac{1}{\sqrt{2}}B^1(\varepsilon_i,\varepsilon_j),\\
B(e'_i,e'_j)=&\frac{1}{\sqrt{2}}B^2(\varepsilon'_i,\varepsilon'_j),
\end{aligned}\right.\\
\left\{\begin{aligned}
B(e_i,\frac{e_0-e'_0}{\sqrt{2}})=&\langle\widetilde{\nabla}_{e_i}(\frac{e_0-e'_0}{\sqrt{2}}),\frac{e_0+e'_0}{\sqrt{2}}\rangle\\
=&\frac{1}{4}e_i(|e_0|^2-|e'_0|^2)+\frac{1}{2}\langle\widetilde{\nabla}_{e_i}e_0,e'_0\rangle
-\frac{1}{2}\langle\widetilde{\nabla}_{e_i}e'_0,e_0\rangle\\
=&0,\\
B(e'_i,\frac{e_0-e'_0}{\sqrt{2}})=&0,
\end{aligned}\right.\\
B(\frac{e_0-e'_0}{\sqrt{2}},\frac{e_0-e'_0}{\sqrt{2}})
=\langle\widetilde{\nabla}_{\frac{e_0-e'_0}{\sqrt{2}}}(\frac{e_0-e'_0}{\sqrt{2}}),\frac{e_0+e'_0}{\sqrt{2}}\rangle=0,\\
B(e_i,e'_j)=\langle\widetilde{\nabla}_{e_i}e'_j,\frac{e_0+e'_0}{\sqrt{2}}\rangle
=-\frac{1}{\sqrt{2}}\langle e'_j,\widetilde{\nabla}_{e_i}e'_0\rangle,
\end{gather*}
where $B^1$ and $B^2$ are the second fundamental forms
of $M\cap(S^{n+1}\times\{y\})\hookrightarrow S^{n+1}\times\{y\}$ and
$(S^{n+1}\times\{y\})\cap M\hookrightarrow \{x\}\times S^{n+1}$, respectively.

\subsection{Principal curvatures}

For an isoparametric hypersurface $M$ in $S^n\times S^n$,
we take an orthogonal frame in a small domain in $M$, with one in it $-\frac{e_0-e_0'}{\sqrt{2}}$ as in Section \ref{2ndfundform}. 
Then the initial data of the ODEs (\ref{Ricatti})-(\ref{curvatureadapted}) under that fra is
$$A=\begin{bmatrix}A'&\\&0\end{bmatrix},\quad
Q=\begin{bmatrix}\frac{1}{2}I_{n-1}&&\\&\frac{1}{2}I_{n-1}&\\&&0\end{bmatrix}.$$
Without loss of generality, we assume $A'$ is diagonal. 
By solving (\ref{Ricatti})-(\ref{curvatureadapted}),
we find the principal curvatures of $M$ are 
$$\lambda_0=0,\quad \lambda_i=-\frac{1}{\sqrt{2}}\cot(\frac{t}{\sqrt{2}}+\theta_i)\ i\ne0.$$
One should reverse the sign when taking opposite normal vector.

Choose a regular point $(x,y)$ of $F$, let $C^x$ (resp. $C^y$) be the great circles in $S^n$ containing $x$ (resp. $y$)
with tangent vector $\pi_1(\widetilde{\mathrm{grad}}\,F)|_x$ (resp. $\pi_2(\widetilde{\mathrm{grad}}\,F)|_y$) at $x$ (resp. $y$).
Since $C^x\times C^y$ is totally geodesic,
the closed geodesic line $C$ starting from $(x,y)$ with initial velocity $e_0+e'_0$ falls in $C^x\times C^y$.
Since $e_0-e'_0$ is a parallel vector field along $C$, which is tangent to $C^x\times C^y$,
the closed geodesic lines starting from any point in $C$ with initial velocity $e_0-e'_0$ falls in $C^x\times C^y$.
Then $C^x\times C^y$ is a disjoint union of these closed geodesic lines,
on each of which $|\widetilde{\mathrm{grad}}\,F|$ takes constant values.
See the following figure for an illustration.
\begin{figure}[H]
\centering
  \includegraphics[height=3cm]{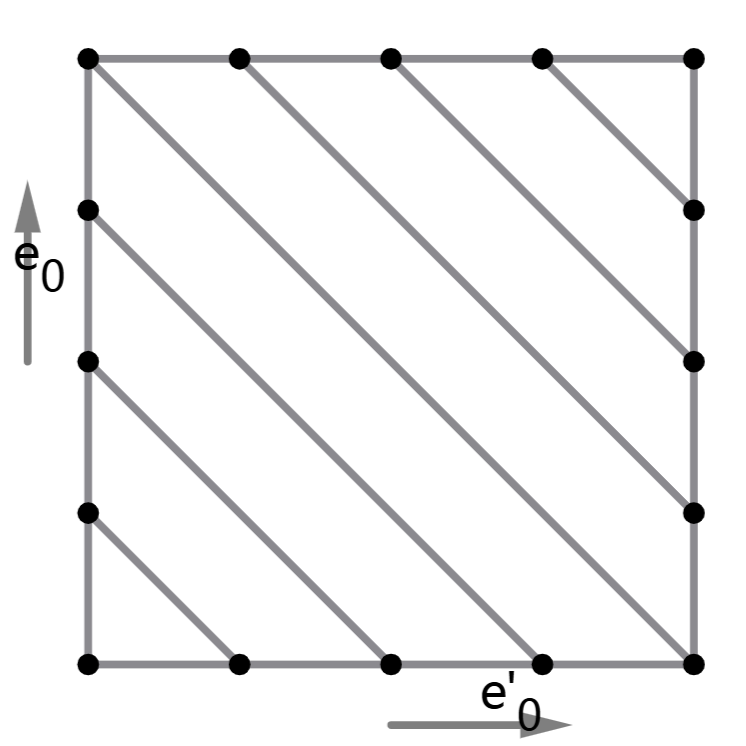}
  \caption*{Singular values of $F|_{C^x\times C^y}$, the case $g=2$}
\end{figure}

The critical points of $F$ are exactly those integral lines of $e_0-e'_0$
passing through some critical point in $C^x\times\{y\}$.
Thus principal curvatures of $M_t$ are
$$-\frac{1}{\sqrt{2}}\cot(\frac{\pi t-(2k-1)\pi}{4g}),\quad k=0,\dots,2g-1.$$
In particular, when $g$ is even, since $F^{-1}(1)=M_+$, $F^{-1}(-1)=M_-$,
those principal curvatures $\to\infty$ as $t\to1$ has multiplicity $m_+$, their values on $M_t$ are
$$-\frac{1}{\sqrt{2}}\cot(\frac{\pi t-(4k+1)\pi}{4g}),\quad k=0,\dots,g-1,$$
those principal curvatures $\to\infty$ as $t\to-1$ has multiplicity $m_-$, their values on $M_t$ are
$$-\frac{1}{\sqrt{2}}\cot(\frac{\pi t-(4k-1)\pi}{4g}),\quad k=0,\dots,g-1.$$
When $g$ is odd, the mean curvature of $M_t$ is
$$H=\frac{-m}{\sqrt{2}}\sum_{k=0}^{2g-1}\cot(\frac{\pi (t-1)}{4g}-\frac{k\pi}{2g})
=-\frac{2gm}{\sqrt{2}}\cot(\frac{\pi(t-1)}{2})=\frac{2gm}{\sqrt{2}}\tan(\frac{\pi t}{2}).$$
When $g$ is even, the mean curvature of $M_t$ is
$$\begin{aligned}
H=&\frac{-m_+}{\sqrt{2}}\sum_{k=0}^{g-1}\cot(\frac{\pi (t-1)}{4g}-\frac{k\pi}{g})
+\frac{-m_-}{\sqrt{2}}\sum_{k=0}^{g-1}\cot(\frac{\pi (t+1)}{4g}-\frac{k\pi}{g})\\
=&-\frac{gm_+}{\sqrt{2}}\cot(\frac{\pi(t+1)}{4})-\frac{gm_-}{\sqrt{2}}\cot(\frac{\pi(t-1)}{4}).
\end{aligned}$$

\subsection{Principal foliations}\label{principalprop}

Principal distributions of any isoparametric hypersurface in $S^n$ are integral and their integral submanifolds
are totally umbilic. Those properties also hold for isoparametric hypersurfaces in $S^n\times S^n$, shown in the following lemma.

\begin{lemma}\label{prinfoli}
Let $M$ be a connected closed isoparametric hypersurface in $S^n\times S^n$,
then each principal distribution of $M$ is integral,
and each of its integral submanifold is a totally unbilic submanifold of $S^n\times S^n$.
\end{lemma}

\begin{proof}
When a principal curvature $-\frac{1}{\sqrt{2}}\cot(\frac{\pi t-(2k-1)\pi}{4g})$ takes $0$, another principal curvature $-\frac{1}{\sqrt{2}}\cot(\frac{\pi t-(2k-2g-1)\pi}{4g})$ takes $\infty$. Since pincipal curvatures are finite on the hypersurface $M$, the principal curvature $0$ of $M$ has multiplicity $1$. 

Denote the principal distribution of the principal curvature $\lambda$ by $T_\lambda$.
Let $B$ be the the second fundamental form of $M$.
Let $e_0$ and $e_0'$ be as in Subsection \ref{ODE}, and $\nu:=\frac{e_0-e_0'}{\sqrt{2}}$, 
which is the principal vector corresponding to the principal curavture $0$.
Define $E_t(x):=\exp_xt\nu$.
Denote the Levi-Civita connection of $\mathbb{R}^{n+1}\times\mathbb{R}^{n+1},S^n\times S^n$ and $M$ by
$\overline{\nabla},\widetilde{\nabla}$ and $\nabla$, respectively.
Let $\widetilde{R}$ be the Riemannian curvature tensor of $S^n\times S^n$, which satisfies
\begin{gather*}
\widetilde{R}(X_1+X_2,Y_1+Y_2)(Z_1+Z_2)=\widetilde{R}^x(X_1,Y_1)Z_1+\widetilde{R}^y(X_2,Y_2)Z_2,\\
\begin{aligned}
&\widetilde{R}(X_1+X_2,Y_1+Y_2,Z_1+Z_2,e_0+e_0')\\
=&\widetilde{R}^x(X_1,Y_1,Z_1,e_0)+\widetilde{R}^y(X_2,Y_2,Z_2,e_0')
=0,
\end{aligned}\\
\begin{aligned}
&\widetilde{R}(X_1+X_2,e_0-e_0',Z_1+Z_2,e_0+e_0')\\
=&\widetilde{R}^x(X_1,e_0,Z_1,e_0)+\widetilde{R}^y(X_2,e_0',Z_2,e_0')\\
=&\langle X_1,Z_1\rangle-\langle X_2,Z_2\rangle,
\end{aligned}
\end{gather*}
where $\widetilde{R}^x$ and $\widetilde{R}^y$ are the Riemannian curvature tensors of $x$-sections and $y$-sections
of $S^n\times S^n$, $X_1,Y_1,Z_1$ tangent vector fields of $S^n\times S^n$ to $x$-sections,
$X_2,Y_2,Z_2$ tangent vector fields of $S^n\times S^n$ to $y$-sections.

For each nonzero principal curvature $\lambda$ of $M$, there exists $t_0\in \mathbb{R}$ such that
$E_{t_0}(M)$ is a focal submanifold $M_+$ of $M$ and $\lim_{t\to t_0}\lambda=\infty$.
Then for each $q\in M_+$, $(E_{t_0})^{-1}(q)$ is a submanifold
of $M$ whose tangent space coincides with $T_\lambda$. In other words,
$(E_{t_0})^{-1}(q)$ is an integral submanifold of $T_\lambda$.
The principal distribution of the principal curvature $0$ is $1$-dimensional, integral.

For any two vector fields $X\in T_\lambda$, and $Y\in T_\mu$, with $|X|=1$ and $\lambda\ne0$, we have
\begin{gather*}
\langle B(\nabla_XX,Y),\nu\rangle=\langle\widetilde{\nabla}_Y\nabla_XX,Y,\nu\rangle
=-\langle\nabla_XX,\widetilde{\nabla}_Y\nu\rangle\\
=\mu\langle\nabla_XX,Y\rangle=-\mu\langle X,\nabla_XY\rangle.
\end{gather*}
On the other hand,
$$\begin{aligned}
\langle B(\nabla_XX,Y),\nu\rangle
=&\langle \nabla^\perp_XB(X,Y)-(\widetilde{\nabla}_XB)(X,Y)-B(X,\nabla_XY),\nu\rangle\\
=&\langle-(\widetilde{\nabla}_YB)(X,X)+\widetilde{R}(Y,X)X,\nu\rangle-\lambda\langle X,\nabla_XY\rangle\\
=&\langle-Y(B(X,X))+2B(\nabla_YX,X)+\widetilde{R}(Y,X)X,\nu\rangle\\
&-\lambda\langle X,\nabla_XY\rangle\\
=&\lambda\langle X,\nabla_XY\rangle+\widetilde{R}(X,Y,X,\nu),
\end{aligned}$$
where we have applied the Codazzi equation to the second equality.
Combine the above two equations to obtain
\begin{equation}\label{eqforumbilic1}
0=(\lambda-\mu)\langle\nabla_XX,Y\rangle+\widetilde{R}(X,Y,X,\nu).
\end{equation}

If $\mu\ne0$, then $\langle Y,e_0\rangle=\langle Y,e_0'\rangle=0$.
Hence $\widetilde{R}(X,Y,X,\nu)=0$ and $\langle\nabla_XX,Y\rangle=0$.

If $\mu=0$, then without loss of generality, we assume $Y=\frac{e_0-e_0'}{\sqrt{2}}$, and (\ref{eqforumbilic1}) becomes
\begin{equation}\label{eqforumbilic2}
0=\lambda\langle\nabla_XX,Y\rangle+\frac{1}{2}|\pi_1X|^2-\frac{1}{2}|\pi_2X|^2.
\end{equation}
Let $M_t:=E_t(M)$, $p_t:=E_t(p)$, $X_t$ the parallel transport of $X$ along the geodesic line $t\mapsto p_t$. 
Without loss of generality, suppose $\lambda=0$ on $M_+$, and $E_T(p)\in M_+$. 
Since the diagonal-property of $A$ in (\ref{Ricatti}) is preserved by parallel transport, 
$X_t$ is parallel to $(E_t)_*X|_p$, also a principal vector, unless $M_t$ is a focal submanifold of $M$. 
Since $X_T\in N_{p_T}M_+$, which is spaned by $\{q,\frac{e_0|_q+e_0'|_q}{\sqrt{2}}\}$, 
where there exists $t$ such that $exp_qt\nu=p_T$, $|\pi_1X_T|=|\pi_2X_T|$.
By the symmetry of two entries, the parallel transport of $X_T$ along $\nu$, $X_t$, satisfies that $|\pi_1X_t|=|\pi_2X_t|$, 
in particular, $|\pi_1X|=|\pi_2X|$.
By (\ref{eqforumbilic2}), $\langle\nabla_XX,Y\rangle=0$.
Therefore, $\nabla_XX$ falls in $T_\lambda$, and any integral submanifold of $T_\lambda$ is totally umbilic.
\end{proof}

\subsection{Homogeneity}

By the classification results of compact reducible linear groups of cohomogeneity $3$ in \cite{uchida1980lineargroup}
which modifies the results of Hsiang and Lawson \cite{hsiang1971minimal},
any linear group action on $\mathbb{R}^{n+1}\times\mathbb{R}^{n+1}$
preserving $S^n\times S^n$,
the restriction of which to $S^n\times S^n$ has codimension-$1$ principal orbits,
is one of
\begin{itemize}
\item $SO(n+1)$ with action $A.(x,y)=(Ax,Ay)$, $n\ge1$,
\item $U(n+1)\times S^1$ with action $(A,a).(z,w)=(Aza,Aw\overline{a})$, $n\ge 1$,
\item $Sp(n+1)\times Sp(1)$ with action $(A,a).(p,q)=(Apa,Aqa^*)$, $n\ge 1$.
\end{itemize}
Obviously, the orbit foliations of the above actions are exactly the level set foliations of
$F_r$,$F_c$,and $F_h$ in Subsection \ref{egr} and Subsection \ref{egcho}, respectively.
Therefore, the level sets foliations of $F_r$,$F_c$,and $F_h$ are
all the homogeneous isoparametric foliations of $S^n\times S^n$.
Any another isoparametric foliation not congruent to one of them with respect to $S^n\times S^n$ is not homogeneous.

\section{Classification of Isoparametric Foliations in $S^n\times S^n$ of Degree $2$ and $4$}\label{classification}

\subsection{The case of bi-degree $(1,1)$}\label{class(1,1)}

Let $F$ be an isoparametric polynomial for $S^n\times S^n$ of bi-degree $(1,1)$
with $\mathrm{im} F|_{S^n\times S^n}=[-1,1]$. By Lemma \ref{isoparabebihomo},
its restriction to each $x$-section and each $y$-section takes $[-1,1]$ in the unit sphere,
and then by Lemma \ref{CM}
$$|\mathrm{grad}^x\,F|^2=|y|^2,\quad|\mathrm{grad}^y\,F|^2=|x|^2.$$
On the other hand, there is a real $(n+1)\times(n+1)$-matrix $A$ such that $F(x,y)=x^tAy$,
and $$|\mathrm{grad}^x\,F|^2=|Ay|^2,\quad|\mathrm{grad}^y\,F|^2=|A^tx|^2.$$
Hence $A$ is an orthogonal matrix.
Change the coordinate via $(Ax',y')=(x,y)$, then $F=x'^tA^tAy'=x'^ty'$.
Therefore, up to congruence with respect to $S^n\times S^n$ and scaling,
$F(x,y)=x^ty$ is the only isoparametric polynomial of bi-degree $(1,1)$,
and all the isoparametric foliations of $S^n\times S^n$ of bi-degree $(1,1)$ are those ones
determined by $F_r$ in Subsection \ref{egr}.

\subsection{The case of bi-degree $(2,2)$}

\begin{lemma}\label{nablasqlemma1}
Let $M$ be an isoparametric hypersurface in $S^n\times S^n$ of bi-degree $( 2,2)$,
then the isoparametric polynomial $F$ determining it
can be taken to satisfy
\begin{equation}\label{varCMeq}
\left\{\begin{aligned}
&|\mathrm{grad}\,F|^2=4F(|x|^2+|y|^2),\\
&\Delta\,F=c(|x|^2+|y|^2).
\end{aligned}\right.
\end{equation}
\end{lemma}

\begin{proof}
Let $F_0$ be the polynomial $F$ in Lemma \ref{isoparabebihomo},
then the restriction of $F_0$ to each $x$-section or $y$-section of $S^n\times S^n$ is
an polynomial satisfies the Cartan-M\"{u}nzner equations.
Apply Cartan-M\"{u}nzner equations to $x$-sections and $y$-sections of $S^n\times S^n$ to obtain
\begin{gather*}
|\mathrm{grad}^x\,F_0|^2=4|x|^2|y|^4,\quad|\mathrm{grad}^y\,F_0|^2=4|y|^2|x|^4,\\
|\mathrm{grad}\,F_0|^2=4(|x|^2+|y|^2)|x|^2|y|^2.
\end{gather*}
Define $F=(F_0+|x|^2|y|^2)/2$, and compute
$$\begin{aligned}
|\mathrm{grad}\,F|^2=&\frac{1}{4}|\mathrm{grad}\,F_0|^2
+2\langle\mathrm{grad}\,\frac{F_0}{2},\mathrm{grad}\,(\frac{|x|^2|y|^2}{2})\rangle+\frac{1}{4}|\mathrm{grad}\,(|x|^2|y|^2)|^2\\
=&(|x|^2+|y|^2)(|x|^2|y|^2+2F_0+|x|^2|y|^2)\\
=&4(|x|^2+|y|^2)F.
\end{aligned}$$
Since $\Delta\,F_0$ and $\Delta |x|^2|y|^2$ are all multiples of $(|x|^2+|y|^2)$, $\Delta\,F$ is also.
\end{proof}

A polynomial of bi-degree $(2,2)$ satisfying (\ref{varCMeq}) is isoparametric for $S^{2n+1}$.
By the classification result of isoparametric functions on spheres,
such a polynomial is congruent to one of the followings with respect to $S^{2n+1}$,
\begin{itemize}
\item $F=\sum_{i=1}^m\langle P_ix,x\rangle^2$, where $x\in\mathbb{R}^{2m+2}$,
$P_i=\psi(e_i)$, $\psi$ is a representation of the Clifford algebra $Cl_m^*$;
\item $F=|x|^4-\sum_{i=1}^m\langle P_ix,x\rangle^2$, where $x\in\mathbb{R}^{2m+2}$,
$P_i=\psi(e_i)$, $\psi$ is a representation of the Clifford algebra $Cl_m^*$;
\item $F=\frac{1}{4}\mathrm{tr}\,(X\overline{X})^2-\frac{1}{4}(\mathrm{tr}\,X\overline{X})^2$,
where $X\in\mathbf{so}(5,\mathbb{F})$, $\mathbb{F}=\mathbb{R}$ or $\mathbb{C}$;
\item $F=\frac{1}{8}(\mathrm{tr}\,X\overline{X})^2-\frac{1}{4}\mathrm{tr}\,(X\overline{X})^2$,
where $X\in\mathbf{so}(5,\mathbb{F})$, $\mathbb{F}=\mathbb{R}$ or $\mathbb{C}$;
\item $F=\frac{1}{4}(|u|^2-|v|^2)^2$ with $u=(x^0,\dots,x^m)$, $v=(x^{m+1},\dots,x^{2n+2})$;
\item $F=|u|^2|v|^2$ with $u=(x^0,\dots,x^m)$, $v=(x^{m+1},\dots,x^{2n+2})$.
\end{itemize}

Therefore, the classification of isoparametric foliations in $S^n\times S^n$ of bi-degree $(2,2)$
is reduced to the classification of
congruence classes of bi-homogeneous $F(Ax)$ with respect to $S^n\times S^n$,
where $F$ is one of the above items and $A\in O(2n+2)$.
In the followings, for each of the above items,
we find out all coordinate changes making it of bi-degree $(2,2)$,
and determining corresponding isoparametric foliations.

\

\textbf{Case I}, $F=\sum_{i=1}^m\langle P_ix,x\rangle^2$, where $P_i=\psi(e_i)$, $\psi$ is a representation of
the Clifford algebra $Cl_m^*$.

Representations in the following discussion are real representions. 

We call a $(2n+2)\times(2n+2)$-matrix quasi skew diagonal
(abbr. QSD) if it is of the form $\begin{bmatrix}&A\\B&\end{bmatrix}$,
where $A$ and $B$ are $(n+1)\times(n+1)$-matrices.

\begin{lemma}\label{functiontomatrix}
Let $F=\sum_{i=1}^m\langle P_ix,x\rangle^2$ be an isoparametric polynomial of FKM-type,
$A$ an orthogonal matrix. Then $F(Ax)$ is bi-homogeneous if and only if $A^tP_iA$ is QSD for any $i$.
\end{lemma}

\begin{proof}
($\Leftarrow$) Obviously.

($\Rightarrow$) If there is a $P_i$ not QSD, then by a orthogonal coordinate change
via an orthogonal matirx $\begin{bmatrix}A&\\&B\end{bmatrix}$, which dose not change whether $P_i$ is QSD,
we factor out $(x^l)^4$ from $\langle P_ix,x\rangle^2$ with positive coefficient for some $l$.
Since we cannot factor out $(x^l)^4$ from other terms with negative coefficient, $F$ is not bi-homogeneous.
\end{proof}

The RHS of Lemma \ref{functiontomatrix} $\iff$ $A^tPA$ is QSD for any $P$ in the Clifford sphere
$\Sigma(P_1,\dots,P_m)=\{a_1P_1+\cdots+a_mP_m|a_1^2+\cdots+a_m^2=1\}$.

\begin{lemma}\label{rotationuniqueness}
Let $\Psi$ be a representation of a Clifford algebra $Cl_m^*$ with each $\Psi(e_i)$ QSD, Let
$$\Sigma=\{\begin{bmatrix}A&\\&B\end{bmatrix}X
\begin{bmatrix}A^t&\\&B^t\end{bmatrix}|A,B\in O(n+1),X\in\Sigma(\Psi(e_1),\dots,\Psi(e_n))\},$$
then the subgroup of $O(2n+2)$ fixing $\Sigma$ is $\begin{bmatrix}O(n+1)&\\&O(n+1)\end{bmatrix}\rtimes
\{\begin{bmatrix}I&\\&I\end{bmatrix},\begin{bmatrix}&I\\I&\end{bmatrix}\}$.
\end{lemma}

\begin{proof}
Obviously, $\begin{bmatrix}O(n+1)&\\&O(n+1)\end{bmatrix}\rtimes
\{\begin{bmatrix}I&\\&I\end{bmatrix},\begin{bmatrix}&I\\I&\end{bmatrix}\}$ does fix $\Sigma$.
Then we prove they are all the elements in $O(2n+2)$ fixing $\Sigma$.

If $\Sigma$ includes $\begin{bmatrix}&X\\X^t&\end{bmatrix}$, then it also includes
$\begin{bmatrix}&PXQ^t\\QX^tP^t&\end{bmatrix}$ for any $P,Q\in O(n+1)$,
and in particular $\begin{bmatrix}&X\\X^t&\end{bmatrix}$ with $X$ a permutation matrix
or a permutation matrix with a $1$ replaced by $-1$.
Hence the linear subspace of $\mathbf{gl}(2n+2)$ generated by $\Sigma$ includes all QSD matrices in $\mathbf{gl}(2n+2)$.

For $P,Q\in O(n+1)$, $\begin{bmatrix}P&\\&Q\end{bmatrix}\Sigma\begin{bmatrix}P^t&\\&Q^t\end{bmatrix}=\Sigma$,
and $$A\Sigma A^{-1}=\Sigma\iff\begin{bmatrix}P&\\&Q\end{bmatrix}A\begin{bmatrix}P^t&\\&Q^t\end{bmatrix}
\Sigma\begin{bmatrix}P&\\&Q\end{bmatrix}A^t\begin{bmatrix}P^t&\\&Q^t\end{bmatrix}=\Sigma.$$
Noting that$$\begin{bmatrix}P&\\&Q\end{bmatrix}\begin{bmatrix}A_{11}&A_{12}\\A_{21}&A_{22}\end{bmatrix}\begin{bmatrix}P^t&\\&Q^t\end{bmatrix}
=\begin{bmatrix}PA_{11}P^t&PA_{12}Q^t\\PA_{21}P^t&QA_{22}Q^t\end{bmatrix},$$
without loss of generality, we can assume that $A_{12}=:\Lambda$ is diagonal, then
\begin{equation*}
\begin{aligned}
&\begin{bmatrix}A_{11}&\Lambda\\A_{21}&A_{22}\end{bmatrix}
\begin{bmatrix}&X\\X^t&\end{bmatrix}
\begin{bmatrix}A_{11}^t&A_{21}^t\\\Lambda&A_{22}^t\end{bmatrix}\\
=&\begin{bmatrix}\Lambda X^tA_{11}^t+A_{11}X\Lambda&\Lambda X^tA_{21}^t+A_{11}XA_{22}^t\\
A_{22}X^tA_{11}^t+A_{21}X\Lambda&A_{22}X^tA_{21}^t+A_{21}XA_{22}^t\end{bmatrix}.
\end{aligned}
\end{equation*}
If RHS is QSD, then $\Lambda X^tA_{11}^t$ is skew-symmetric
for any $\begin{bmatrix}&X\\X^t&\end{bmatrix}$ in the linear subspace of $\mathbf{gl}(2n+2)$ generated by $\Sigma$.
If $\Lambda=0$, then $A\in O(2n+2),\Longrightarrow A_{11}$ and $A_{22}$ are orthogonal,
$\Longrightarrow A=\begin{bmatrix}A_{11}&\\&A_{22}\end{bmatrix}$.
If $\Lambda\ne0$, then taking $X=E_{ij}$ ($i,j=1,\dots,n+1$) in $\Lambda X^tA_{11}^t+A_{11}X\Lambda=0$,
we find that $A_{11}=0$, and again from $A\in O(2n+2)$ we know that
$A=\begin{bmatrix}&\Lambda\\A_{21}&\end{bmatrix}$.
\end{proof}

From this lemma we know that for each represnetation of $Cl_m^*$,
there is at most one orthogonal coordinate
change making $\psi(e_1),\dots,\psi(e_m)$ QSD simultaneously up to congruence with respect to $S^n\times S^n$.

Let $\psi$ be a representation of $Cl_m^*$,  we say $\psi$ is quasi-skew diagonalizable  (abbr. QSD) if
there exists an orthogonal coordinate change making $\psi(e_1),\dots,\psi(e_m)$ QSD simultaneously.

\begin{lemma}\label{QSD}
Let $\psi_m$ be the irreducible representation of $Cl_m^*$ given in Table \ref{tableeg2},
$\psi_m'$ the irreducible representation of $Cl_m^*$ given via $\psi_m'(e_i)=-\psi_m(e_i)$.
\begin{itemize}
\item[1.] An irreducible representation of $Cl_m^*$ is QSD
if and only if $m\equiv4,6,7,$ or $8\mod8$;
\item[2.] if $m\equiv4,6,7,$ or $8\mod8$ then $k\psi_m$ is QSD;
\item[3.] if $m\equiv2$ or $3\mod8$ then $k\psi_m$ is QSD if and only if $k$ is even;
\item[4.] if $m\equiv1$ or $5\mod8$ then $k_1\psi_m+k_2\psi_m'$ is QSD if and only if $k_1=k_2$.
\end{itemize}
\end{lemma}

\begin{proof}
If $\delta^*(m+1)=\delta^*(m)$, then by checking Table \ref{tableeg1} and Table \ref{tableeg2}, we know that
$$\psi(e_1)=\begin{bmatrix}&I\\I&\end{bmatrix},\quad
\psi(e_i)=\begin{bmatrix}&\phi_{m-1}(e_{i-1})\\ -\phi_{m-1}(e_{i-1})&\end{bmatrix}$$
gives an irreducible representation of $Cl_m^*$ which is QSD.
On the other hand, if $\delta^*(m+1)>\delta^*(m)$, and there is an irreducible representation of $Cl_m^*$ given by
$$e_1\mapsto\begin{bmatrix}&Q_1\\R_1&\end{bmatrix},\dots,e_m\mapsto\begin{bmatrix}&Q_m\\R_m&\end{bmatrix},$$
then
$$e_1\mapsto\begin{bmatrix}&Q_1\\R_1&\end{bmatrix},\dots,e_m\mapsto\begin{bmatrix}&Q_m\\R_m&\end{bmatrix},
e_{m+1}\mapsto\begin{bmatrix}I&\\&-I\end{bmatrix}$$
gives a $\delta^*(m)$-dimensional irreducible representation of $Cl_{m+1}^*$,
a contradiction. Since $\delta^*(m+1)=\delta^*(m)$ if and only if $m\equiv4,6,7,8\mod8$, 1.) and 2.) hold.

If $m\equiv1,2,3,5(\mathrm{mod}\,8)$ and $\psi$ is a QSD representation of $Cl_m^*$,
then by choosing a suitable basis, we can make sure that
$$P_1=\begin{bmatrix}&E_1\\-E_1&\end{bmatrix},\dots,P_{m-1}=\begin{bmatrix}&E_{m-1}\\-E_{m-1}&\end{bmatrix},
P_m=\begin{bmatrix}&I\\I&\end{bmatrix},$$
where $P_i=\psi(e_i)$, and $\{e_1\mapsto E_1,\dots,e_{m-1}\mapsto E_{m-1}\}$ stands for a representation of $Cl_{m-1}$.
Since $\dim E_i$ is a multiple of $\delta(m-1)=\delta^*(m)$,
$\dim P_i$ is an even multiple of $\delta^*(m)$, i.e. $\psi$ is the direct sum of an even number of irreducible representations.
If in addition $m\equiv1,5(\mathrm{mod}\,8)$,
since $P_1P_2\cdots P_m$ is QSD, $\mathrm{tr}\,P_1P_2\cdots P_m=0$
and that holds only when the representation is $k\psi_m+k\psi_m'$.
On the other hand, (\ref{fulleg}) are examples of $k\psi_m$ with $m\equiv2$ or $3(\mathrm{mod}\,8)$,
and examples of $k\psi_m+k\psi_m'$ with $m\equiv1$ or $5(\mathrm{mod}\,8)$. Therefore, 3.) and 4.) hold.
\end{proof}

From Lemma \ref{QSD}, we know that isoparametric polynomials for $S^n\times S^n$ of bi-degree $(2,2)$
given in Case I are exactly those given by (\ref{fulleg}) and (\ref{notfulleg}).

If $\psi$ is an $(n+1)$-dimensional (resp. $2(n+1)$-dimensional) representation of $Cl_m^*$
when $m\cong1,2,3,5\mod8$ (resp. $m\cong4,6,7,8\mod8$), then $m_+$ and $m_-$ corresponding to $F^*$
given by (\ref{fulleg}) (resp. (\ref{notfulleg})), satisfying $(m_+,m_-)=(m-1,n-m)$,
when $n\ge m$, i.e., $(m,n)\ne(2,1),(4,3)$ nor $(8,7)$.
Therefore, the foliation determined by $F^*$ is of bi-degree $(1,1)$ if and only if $(m,n)=(3,3),(7,7)$ or $m=1$.

\

\textbf{Case II}, $F=|x|^4-\sum_{i=1}^m\langle P_ix,x\rangle^2$, where $P_i=\psi(e_i)$, $\psi$ is a representation of
the Clifford algebra $Cl_m^*$.

Without loss of generality, we assume $P_1=\begin{bmatrix}I&\\&-I\end{bmatrix}$,
then $P_2,\dots,P_m$ are all $QSD$, since $P_1P_i+P_iP_1=0$ for $i=2,\dots,m$.
Then $F=4|x^{(1)}|^2|x^{(2)}|^2-\sum_{i=2}^m\langle P_ix,x\rangle^2$,
which determine the same isoprametric foliation of $S^n\times S^n$ as $F=\sum_{i=2}^m\langle P_ix,x\rangle^2$.

Therefore, isoparametric polynomials functions for $S^n\times S^n$ of bi-degree $(2,2)$ given in Case II
are included in those given in Case I.

\

\textbf{Case III}, $F=\frac{1}{4}\mathrm{tr}\,(X\overline{X})^2-\frac{1}{4}(\mathrm{tr}\,X\overline{X})^2$,
where $X\in\mathbf{so}(5,\mathbb{F})$, $\mathbb{F}=\mathbb{R}$ or $\mathbb{C}$.

Suppose eigenvalues of $X$ are $a,-a,b,-b$,
then $F=\frac{1}{4}(|a|^2-|b|^2)^2$. The level set $\{F=0\}$ is the cone of those matrices with eigenvalues
$a,-a,a,-a$, where $a\in\mathrm{i}\mathbb{R}$ when $\mathbb{F}=\mathbb{R}$, $a\in\mathbb{C}$ when $\mathbb{F}=\mathbb{C}$.
If $F$ is of bi-degree $(2,2)$, then its null set must contain
a $5$-dimensional subspace when $\mathbb{F}=\mathbb{R}$, a $10$-dimensional subspace when $\mathbb{F}=\mathbb{C}$.
Since the cone $\{F=0\}$ is $\mathrm{ISO}(F)$-transitive, the point
$$X_0=\begin{bmatrix}&1&&&\\-1&&&&\\&&&1&\\&&-1&&\\&&&&0\end{bmatrix}$$
must be contained in a $10$-dimensional (resp. $5$-dimensional) subspace on which $F$ takes $0$
when $\mathbb{F}=\mathbb{C}$ (resp. $\mathbb{R}$).

When $\mathbb{F}=\mathbb{C}$, an element in the tangent space of $\{F=0\}$ at $X_0$,
which is perpendicular to $X_0$, is of the form
$$\begin{bmatrix}
&x\mathrm{i}&z'&z&w_1\\
-x\mathrm{i}&&\overline{z}&-\overline{z'}&w_2\\
-z'&-\overline{z}&&y\mathrm{i}&w_3\\
-z&-\overline{z'}&-y\mathrm{i}&&w_4\\
-w_1&-w_2&-w_3&-w_4&
\end{bmatrix},$$
where $x,y\in\mathbb{R}$, $z,z',w_1,w_2,w_3,w_4\in\mathbb{C}$.
By direct computation, we have
\begin{equation*}\begin{aligned}
F(X_0+tX)=&4t^4|(x+y)\overline{z}\mathrm{i}+w_3\overline{w_1}+w_2\overline{w_4}|^2\\
&+4t^4|-(x+y)\overline{z'}\mathrm{i}+w_4\overline{w_1}+w_2\overline{w_3}|^2\\
&+t^4(x^2+y^2+|w_1|^2+|w_2|^2-|w_3|^2-|w_4|^2)^2.
\end{aligned}\end{equation*}
To make sure the line $X_0+\mathbb{R}X$ is in the cone $\{F=0\}$, $X$ should satisfies three independent equations
\begin{equation*}\left\{\begin{aligned}
&(x+y)\overline{z}\mathrm{i}+w_3\overline{w_1}+w_2\overline{w_4}=0,\\
&-(x+y)\overline{z'}\mathrm{i}+w_4\overline{w_1}+w_2\overline{w_3}=0,\\
&x^2+y^2+|w_1|^2+|w_2|^2-|w_3|^2-|w_4|^2=0.
\end{aligned}\right.\end{equation*}
Hence the dimension of the space of $X$ with $X_0+\mathbb{R}X$ in the cone $\{F=0\}$
is $\le14-3\times2+1=9$.

When $\mathbb{F}=\mathbb{R}$, an element in the tangent space of $\{F=0\}$ at $X_0$,
which is perpendicular to $X_0$, is of the form
$$\begin{bmatrix}
&&x'&x&y_1\\
&&x&-x'&y_2\\
-x'&-x&&&y_3\\
-x&x'&&&y_4\\
-y_1&-y_2&-y_3&-y_4&
\end{bmatrix},$$
where $x,x',y_1,y_2,y_3,y_4\in\mathbb{R}$.
By direct computation, we have
\begin{equation*}\begin{aligned}
F(X_0+tX)=&4t^4(y_1y^3-y_2y_4)^2
+4t^4(y_1y_4-y_2y_3)^2\\
&+t^4(|y_1|^2+|y_2|^2-|y_3|^2-|y_4|^2)^2\\
=&t^4(|y_1|^2+|y_2|^2+|y_3|^2+|y_4|^2)^2.
\end{aligned}\end{equation*}
Therefore, the dimension of the space of $X$ with $X_0+\mathbb{R}X$ in the cone $\{F=0\}$
is $\le2+1=3$.

Therefore, in Case III, $F$ cannot be made of bi-degree $(2,2)$ through a coordinate change.

\

\textbf{Case IV}, $F=\frac{1}{8}(\mathrm{tr}\,X\overline{X})^2-\frac{1}{4}\mathrm{tr}\,(X\overline{X})^2$,
where $X\in\mathbf{so}(5,\mathbb{F})$, $\mathbb{F}=\mathbb{R}$ or $\mathbb{C}$.

Suppose eigenvalues of $X$ are $a,-a,b,-b$,
then $F=|a|^2|b|^2$. The level set $\{F=0\}$ is the cone of those matrices with eigenvalues
$a,-a,0,0$, where $a\in\mathrm{i}\mathbb{R}$ when $\mathbb{F}=\mathbb{R}$, $a\in\mathbb{C}$ when $\mathbb{F}=\mathbb{C}$.
If $F$ is of bi-degree $(2,2)$, then its null set must contain
a $5$-dimensional subspace when $\mathbb{F}=\mathbb{R}$, a $10$-dimensional subspace when $\mathbb{F}=\mathbb{C}$.
Since the cone $\{F=0\}$ is $\mathrm{ISO}(F)$-transitive, the point
$$X_0=\begin{bmatrix}&1&&&\\-1&&&&\\&&0&&\\&&&0&\\&&&&0\end{bmatrix}$$
must be contained in a $10$-dimensional (resp. $5$-dimensional) subspace on which $F$ takes $0$
when $\mathbb{F}=\mathbb{C}$ (resp. $\mathbb{R}$).

When $\mathbb{F}=\mathbb{C}$, an element in the tangent space of $\{F=0\}$ at $X_0$,
which is perpendicular to $X_0$, is of the form
$$X=\begin{bmatrix}
&x\mathrm{i}&z_1&z_2&z_3\\
-x\mathrm{i}&&z_4&z_5&z_6\\
-z_1&-z_4&&&\\
-z_2&-z_5&&&\\
-z_3&-z_6&&&
\end{bmatrix},\text{ where $x\in\mathbb{R}$, $z_1,\dots,z_6\in\mathbb{C}$.}$$
When $\mathbb{F}=\mathbb{R}$, an element in the tangent space of $\{F=0\}$ at $X_0$,
which is perpendicular to $X_0$, is of the form
$$X=\begin{bmatrix}
&&x_1&x_2&x_3\\
&&x_4&x_5&x_6\\
-x_1&-x_4&&&\\
-x_2&-x_5&&&\\
-x_3&-x_6&&&
\end{bmatrix},\text{ where $x_1,\dots,x_6\in\mathbb{R}$.}$$
In both cases, to make sure $\mathrm{rank}(X_0+tX)=2$ for all $t\in\mathbb{R}$, $X$ must be of the form
$$\begin{bmatrix}
&a&b_1&b_2&b_3\\
-a&&&&\\
-b_1&&&&\\
-b_2&&&&\\
-b_3&&&&
\end{bmatrix}
\text{ or }
\begin{bmatrix}
&a&&&\\
-a&&b_1&b_2&b_3\\
&-b_1&&&\\
&-b_2&&&\\
&-b_3&&&
\end{bmatrix},$$
where $a\in\mathrm{i}\mathbb{R}$ when $\mathbb{F}=\mathbb{C}$, $a=0$ when $\mathbb{F}=\mathbb{R}$,
$b_1,b_2,b_3\in\mathbb{F}$.
Hence the dimension of the space of $X$ with $X_0+\mathbb{R}X$ in the cone $\{F=0\}$
is $\le7+1=7$ when $\mathbb{F}=\mathbb{C}$, $\le3+1=4$ when $\mathbb{F}=\mathbb{R}$.

Therefore, in Case IV, $F$ cannot be made of bi-degree $(2,2)$ through a coordinate change.

\

\textbf{Case V}, $F=(|u|^2-|v|^2)^2$ with $u=(x^0,\dots,x^m)$, $v=(x^{m+1},\dots,x^{2n+2})$.

If there is a coordinate change making $F$ of bi-degree $(2,2)$, then there exist
two $(n+1)$-dimensional planes perpendicular to each other on which $F$ takes $0$.
Since the null set of $F$ is $\{(u,v)||u|^2=|v|^2\}$, 
such two planes exist only if $m=n$ and the coordinate change
can only take $\left\{\begin{aligned}x=u+v,\\y=u-v,\end{aligned}\right.$ up to congruence with respect to $S^n\times S^n$.
Then $F=\langle x,y\rangle^2$, determines the same isoparametric foliation as $F_r$,
which is of bi-degree $(1,1)$.

\

\textbf{Case VI}, $F=4|u|^2|v|^2$ with $u=(x^0,\dots,x^m)$, $v=(x^{m+1},\dots,x^{2n+2})$.

As in Case V, there exists a coordinate change making $F$ of bi-degree $(2,2)$  only if there exist
two $(n+1)$-dimensional planes perpendicular to each other on which $F$ takes $0$.
Since the null set of $F$ is $\{(u,v)|u=0\text{ or }v=0\}$, 
that happens only when $m=n$ and the coordinate change takes $\left\{\begin{aligned}x=u,\\y=v,\end{aligned}\right.$
up to congruence with respect to $S^n\times S^n$.
Then $F=|x|^2|y|^2$, which is trivial.

\

In summary, all the isoparametric foliations of $S^n\times S^n$ of bi-degree $(2,2)$
are those non-degenerate ones determined by $F$ in Case I,
i.e., those Clifford-type examples given by (\ref{fulleg}) and (\ref{notfulleg}),
except for the case $(m,n)=(2,1),(3,3),(4,3),(7,7),(8,7)$ or $m=1$.
In fact, the following lemma tells us that Cui's formula (\ref{cuieg}) is a unified expression of
all the bi-homogeneous isoparametric polynomials of bi-degree $(2,2)$ for $S^n\times S^n$.

\begin{lemma}\label{cliffordegeq}
For each isoparametric poolynomial $F$ for $S^n\times S^n$ given by formula (\ref{fulleg}) or (\ref{notfulleg}),
there exists one and only one isoparametric polynomial for $S^n\times S^n$ given by formula (\ref{cuieg})
congruent to it with respect to $S^n\times S^n$.
\end{lemma}

\begin{proof}
We denote by $F_{m,k}$ the isoparametric polynomial given in the form of (\ref{cuieg})
if $\phi$ is a direct sum of $k$ irreducible representations of $Cl_{m-1}$,
$F^*_{m,k}$ the isoparametric polynomial given in the form of (\ref{fulleg}) or (\ref{notfulleg})
if $\psi\simeq k\psi_m$ when $m\equiv4,6,7,8\mod8$,
$\psi\simeq 2k\psi_m$ when $m\equiv2,3\mod8$,
and $\psi\simeq k\psi_m+k\psi_m'$ when $m\equiv1,5\mod8$.
By the discussion in Subsection \ref{exampleclifford} and Lemma \ref{QSD}, we know for each $(m,k)$,
there is one and only one $F_{m,k}$, one and only one $F^*_{m,k}$.

For each $(m,k)$, by the above classification result, $F_{m,k}$ is congruent to some $F^*_{m',k'}$.
Since $F_{m,k}$ is defined on $\mathbb{R}^{2k\delta(m-1)}\times\mathbb{R}^{2k\delta(m-1)}$, $\Delta\,F_{m,k}=2m(|x|^2+|y|^2)$, and $F^*_{m,k}$ is also,  $F_{m,k}$ is congruent to $F^*_{m',k'}$ if and only if $(m',k')=(m,k)$.
\end{proof}

\bibliography{Isoparametric_hypersurface.bbl}

\begin{flushleft}
			\medskip\noindent
			\begin{tabbing}
				XXXXXXXXXXXXXXXXXXXXXXXXXX*\=\kill
				Teng Wang\\
				School of Mathematical Sciences,\\
 University of Science and Technology of China,\\
				96 Jinzhai Road, Hefei, 230026, Anhui Province, China\\
				
				E-mail: wt4723@mail.ustc.edu.cn
				
			\end{tabbing}
		\end{flushleft}

\end{document}